\documentclass[a4paper,12pt]{amsart}
\usepackage{amssymb,amscd,amsfonts,amsbsy,nicefrac,mathtools}
\usepackage[utf8]{inputenc}
\usepackage{amsmath,amsthm,amssymb,amsfonts,enumitem}
\usepackage{mathrsfs}
\usepackage[nobysame]{amsrefs}
\usepackage[pdftex]{graphicx,color}

\usepackage{graphicx}
\usepackage{ marvosym }
\usepackage{scalerel}
\usepackage{bm}

\usepackage[a4paper,tmargin=3cm,bmargin=3cm,lmargin=3cm,rmargin=3cm]{geometry}
\usepackage{enumitem}
\newtheorem{thm}{Theorem}[section]
\newtheorem{lem}[thm]{Lemma}

\newtheorem{cor}[thm]{Corollary}
\newtheorem{prop}[thm]{Proposition}
\newtheorem{defn}[thm]{Definition}
\newtheorem{rem}[thm]{Remark}

\def\R{\mathbb{R}}

\def\la{\langle}
\def\ra{\rangle}

\def\2q{{\frac{2}{|B|}}}

\newcommand{\N}{\mathbb{N}}

\newcommand{\bmo}{\mathrm{bmo}}

\newcommand{\sgn}{\operatorname{sgn}}

\newcommand{\supp}{\operatorname{supp}}

\newcommand\fat[1]{\ThisStyle{\ooalign{%
  \kern.46pt$\SavedStyle#1$\cr\kern.33pt$\SavedStyle#1$\cr%
  \kern.2pt$\SavedStyle#1$\cr$\SavedStyle#1$}}}

\renewcommand{\S}{{\mathscr{S}}}

\renewcommand{\leq}{\leqslant}
\renewcommand{\geq}{\geqslant}

\newcommand{\phase}{\varphi}

\newcommand{\m}[1]{\begin{equation*}
\begin{split}
#1
\end{split}
\end{equation*}}
\newcommand{\nm}[2]{\begin{equation}\label{#1}
\begin{split}
#2
\end{split}
\end{equation}}

\newcommand{\abs}[1]{\left|#1\right|}
\newcommand{\set}[1]{\left\{#1\right\}}
\newcommand{\brkt}[1]{\left(#1\right)}
\newcommand{\jap}[1]{\left\langle #1\right\rangle}
\newcommand{\norm}[1]{\left\Vert#1\right\Vert}

\newcommand{\dd}{\,\mathrm{d} }
\newcommand{\ddd}{\,\text{\rm{\mbox{\dj}}}}
\renewcommand{\d}{\partial}

\begin{document}

\title[Local and global boundedness] {Local and global estimates for hyperbolic equations in Besov-Lipschitz and Triebel-Lizorkin spaces}
\author[A.~Israelsson]{Anders Israelsson}
\address{Department of Mathematics, Uppsala University, SE-751 06 Uppsala, Sweden}
\email{anders.israelsson@math.uu.se}
\author[S.~Rodr\'iguez-L\'opez]{Salvador Rodr\'iguez-L\'opez}
\address{Department of Mathematics, Stockholm University, SE-106 91 Stockholm, Sweden}
\email{s.rodriguez-lopez@math.su.se}

\author[W.~Staubach]{Wolfgang Staubach}
\address{Department of Mathematics, Uppsala University, SE-751 06 Uppsala, Sweden}
\email{wulf@math.uu.se}
\date{\today}
\thanks{The second author is partially supported by the Spanish Government grant MTM2016-75196-P.
The third author is partially supported by a grant from the Crafoord foundation and by a grant from G. S. Magnusons fond, grant number MG2015-0077.}

\subjclass[2010]{35S30, 42B20, 35L05, 35L15, 42B35.}
\keywords{Besov-Lipschitz spaces, Triebel-Lizorkin spaces, Fourier integral operators, Hyperbolic equations}
\begin{abstract}
In this paper we establish optimal local and global Besov-Lipschitz and Triebel-Lizorkin estimates for the solutions to linear hyperbolic partial \linebreak differential equations. These estimates are based on local and global estimates for Fourier integral operators that span all possible scales (and in particular both \linebreak Banach and quasi-Banach scales) of Besov-Lipschitz spaces $B^s_{p,q}(\R^n)$, and certain Banach and quasi-Banach scales of Triebel-Lizorkin spaces $F^s_{p,q}(\R^n)$.
\end{abstract}

\maketitle

\section{Introduction}
\noindent 
Estimates for the solution of linear hyperbolic partial differential equations, in function spaces other than the $L^p$ spaces, go back to the 1970's. In this context we would like to mention a couple of results, that although not directly relevant to obtaining the results of the current paper, constitute examples of estimates in function spaces that are of interest here, namely the Besov-Lipschitz and Triebel-Lizorkin spaces.\\
Consider the following Cauchy problem for the wave equation in $\R^{n+1}$,

\begin{equation}\label{the first wave}
     \left\{ \begin{array}{lll}
         \partial^2_t u (t,x)-\Delta_x u(t,x)=0, & t\not=0, \, x\in\R^n, \\
         u(0,x)=f_0 (x),\\
         \partial_t u(0,x)= f_1(x).
      \end{array} \right.
\end{equation}\hspace*{1cm}

In \cite{Brenner} P. Brenner showed that for a fixed time $\tau>0$ the solution to this problem
verifies the estimate
\begin{equation}\label{brenners global besov estimate for the wave equation}
\Vert u (\tau,\cdot)\Vert_{B^{s}_{p,q}(\R^n)}\leq C_\tau \left ( \Vert f_0\Vert_{B^{s+\nu}_{p',q}(\R^n)}+ \Vert f_1\Vert_{B^{s+\nu-1}_{p',q}(\R^n)} \right ),
\end{equation}
where $s\in \R$, $p\in [2,\infty)$, $\displaystyle p'=\frac{p}{p-1}$, $q\in [1,\infty]$, and \linebreak$\displaystyle (n+1)\abs{\frac{1}{p}-\frac{1}{2}}\leq \nu\leq 2n\abs{\frac{1}{p}-\frac{1}{2}}$.\\

\noindent In \cite {Kapitanskii} L. V. Kapitanski\u\i, extended and improved the results of Brenner to the range $p\in [2,\infty]$ and $\displaystyle(n-1)\abs{\frac{1}{p}-\frac{1}{2}}\leq \nu\leq n\abs{1-\frac{2}{p}}.$ In fact Kapitanski\u\i's result also applies to more general variable coefficient second order strictly hyperbolic equations, and also is valid in the realm of Triebel-Lizorkin spaces for the same range of parameters.\\

Later, J. Ginibre and G. Velo \cite{GV} established Strichartz-type estimates for homogeneous Besov-Lipschitz and Triebel-Lizorkin spaces which are useful in the applications to non-linear hyperbolic problems.\\

\noindent However, the pioneering results of Brenner's were achieved by establishing $L^p\to L^q$ estimates for a class of Fourier integral operators that appear naturally in the construction of solutions (or parametrises) for strictly hyperbolic partial differential equations.\\

\noindent The next breakthrough was made in \cite{SSS}, where A. Seeger, C. Sogge and E. Stein showed that for every smooth spatial cut-off function $\chi$ one has the estimate

\begin{equation}\label{SSS local sobolev estimate for the wave equation}
\Vert \chi\,u (\tau,\cdot)\Vert_{H^{s,p}(\R^n)}\leq C_\tau\left ( \Vert f_0\Vert_{H^{s+\nu,p}(\R^n)}+ \Vert f_1\Vert_{H^{s+\nu-1, p}(\R^n)} \right ),
\end{equation}
for $s\in\R$, $\displaystyle\nu=(n-1)\abs{\frac{1}{p}-\frac{1}{2}}$ and $p\in (1,\infty).$\\

As a consequence of this, one has
\begin{equation*}
\Vert \chi\,u (\tau,\cdot)\Vert_{B^{s}_{2,2}(\R^n)}\leq C_\tau\left ( \Vert f_0\Vert_{B^{s}_{2,2}(\R^n)}+ \Vert f_1\Vert_{B^{s-1}_{2, 2}(\R^n)} \right ),
\end{equation*}
\noindent with $s\in\R$. Moreover, in \cite{SSS} it was also proven that
\begin{equation*}
\Vert \chi\,u (\tau,\cdot)\Vert_{B^{s}_{\infty,\infty}(\R^n)}\leq C_\tau\left ( \Vert f_0\Vert_{B^{s+\nu}_{\infty,\infty}(\R^n)}+ \Vert f_1\Vert_{B^{s+\nu-1}_{\infty,\infty}(\R^n)} \right ),
\end{equation*}
for $\displaystyle \nu=\frac{n-1}{2}$. This is of course nothing but the Lipschitz space estimate in \cite{SSS}. \\

In this paper we establish the global estimate
\begin{equation*}
\Vert u (\tau,\cdot)\Vert_{B^{s}_{p,q}(\R^n)}\leq C_\tau \left ( \Vert f_0\Vert_{B^{s+\nu}_{p,q}(\R^n)}+ \Vert f_1\Vert_{B^{s+\nu-1}_{p,q}(\R^n)} \right ),
\end{equation*}
for the solution to \eqref{the first wave}, where $s\in \R$, $\displaystyle p\in \left (\frac{n}{n+1},\infty \right ]$, $q\in (0,\infty]$, and $\displaystyle \nu= (n-1)\left |\frac{1}{p}-\frac{1}{2}\right |,$ where the ranges of these parameters are optimal.\\

\noindent Moreover we also show that the local version of the above estimate is valid for $s\in \R$, $p\in (0,\infty]$ and $q\in (0,\infty)$. Furthermore we show the following global estimate for the Triebel-Lizorkin spaces
\begin{equation*}
\Vert u (\tau,\cdot)\Vert_{F^{s}_{p,q}(\R^n)}\leq C_\tau \left ( \Vert f_0\Vert_{F^{s+\nu}_{p,q}(\R^n)}+ \Vert f_1\Vert_{F^{s+\nu-1}_{p,q}(\R^n)} \right ),
\end{equation*}
where $s\in \R$, $\displaystyle p\in \left (\frac{n}{n+1},\infty \right ]$, $\min (2,p)\leq q\leq \max (2,p)$, and $\displaystyle \nu= (n-1)\left|\frac{1}{p}-\frac{1}{2}\right |.$

At the local level, we can improve the range of $p$ in the above estimate to $(0,\infty)$.
 However if one assumes that $\displaystyle \nu<- (n-1)\abs{\frac{1}{p}-\frac{1}{2}}$, then the range of the Triebel-Lizorkin estimate above is improved to the optimal range $p, q\in (0,\infty]$ in the local \smallskip case, and $\displaystyle p\in \left (\frac{n}{n+1},\infty \right ]$, $q\in (0,\infty]$ in the global case. Moreover, as was done in \cite{Kapitanskii} and \cite{SSS}, we also establish similar estimates for more general variable coefficient hyperbolic PDEs. \\

All of these results are achieved through proving sharp local and global estimates for Fourier integral operators of the form
\begin{equation}\label{definition of FIO}
    T^\varphi_ a f(x)=  \int_{\R^n} a(x,\xi) \, e^{i\varphi(x,\xi)}\, \widehat{f}(\xi)\, \ddd \xi,
\end{equation}
with smooth amplitudes $a(x, \xi)\in S^m (\R^n)$ (see Definition \ref{symbol class Sm}), on Besov-Lipschitz and Triebel-Lizorkin spaces. The interest in these spaces stems from the fact that they contain spaces such as Lebesgue spaces, Lipschitz spaces (H\"older spaces), Sobolev spaces, Hardy spaces and BMO spaces, as special cases. Moreover these spaces also contain scales that are quasi-Banach and indeed one of the purposes of the this paper is to extend the estimates for the solutions of the wave equation to the quasi-Banach setting. It turns out that in the context of global estimates for Fourier integral operators, the restriction for $p$ being in $\displaystyle \left (\frac{n}{n+1},\infty \right ]$, is sharp for the validity of global estimates, since we can produce counter-examples to the global boundedness of the Fourier integral operators for $\displaystyle p \in \left (0, \frac{n}{n+1} \right ]$. However, if one is looking for local estimates, as in for example \cite{SSS}, then we show that in that case the range of the $p$'s can indeed be improved to the full range $(0, \infty]$. We should also mention that although optimal local $L^p$ estimates for Fourier integral operators are by now classical (see \cite{SSS}), the optimal global $L^p$ estimates for these operators are rather recent (see the papers by S. Coriasco and M. Ruzhansky \cite{CR1}, \cite{CR2}, and M. Ruzhansky and M. Sugimoto \cite{Ruzhansky-Sugimoto}). Other global boundedness results for classes of Fourier integral operators on scales of relevant functional spaces, namely, the modulation spaces, have been proved in the work of  F. Concetti, G. Garello and J. Toft \cite{CGT2} and in the paper by E. Cordero, F. Nicola and L. Rodino \cite{CNR2}) (see also \cite{CGT1} and \cite{CNR1}, for similar results on $L^2$ and the $\mathcal{F}L^p(\mathbb{R}^d)_{\mathrm{comp}}$ spaces, respectively). Another collection of recent and interesting results regarding global boundedness of Fourier integral operators, that goes beyond \cite{Ruzhansky-Sugimoto} and encompass more general amplitudes and homogeneous of degree one phase functions is that of A. Hassell, P. Portal and J. Rozendaal \cite{HPR}. In \cite{memoirs}  D. Dos Santos-Ferreira and W. Staubach proved local and global $L^p$ estimates for Fourier integral operators with amplitudes that are merely bounded in the spatial variables and those results were extended by S. Rodr\'iguez-L\'opez and W. Staubach \cite{JFA} to operators with amplitudes belonging to $L^p$ in their spatial variables. Some attempts in establishing estimates in Triebel-Lizorkin space were also made in \cite{memoirs}. However those estimates didn't yield the results obtained here, due to the fact that they were based on vector-valued inequalities for Fourier integral operators which were in turn based on the weighted norm inequalities proven in that paper. The weighted inequalities in \cite{memoirs} require a sharp order of decay, which is worse than the optimal expected order of decay for the validity of Triebel-Lizorkin estimates.\\

The paper is organised as follows; in Section \ref{subsection:definitions} we recall some definitions, facts and results from microlocal and harmonic analysis that will be used throughout the paper. In Section \ref{SSS section} we decompose the Fourier integral operators into certain pieces and establish the basic kernel estimates for these pieces. The kernel estimates obtained here are also valid for non-regular amplitudes. These kernel estimates are used in the proof of the regularity in both Besov-Lipschitz and Triebel-Lizorkin spaces in the later sections, for Fourier integral operators with regular amplitudes. In Section \ref{RS globalisation} we describe the transference of local to global regularity of Fourier integral operators due to M. Ruzhansky and M. Sugimoto, and how it can be fit into our setting. In Section \ref{boundedness in Besov} we prove the optimal local and global boundedness of Fourier integral operators on all possible scales of Besov-Lipschitz spaces (Theorem \ref{thm:local and global BL}), however one of the intermediate results (Proposition \ref{prop:JFA}) deals with non-smooth amplitudes. In Section \ref{boundedness on Triebel} we deal with the regularity problem in certain scales of Triebel-Lizorkin spaces and obtain optimal results for those scales (Theorem \ref{thm:local and global TL complement}). However, we also show that if the order of the operator is just below the critical threshold, then the Triebel-Lizorkin regularity can be extended to all possible scales of the Triebel-Lizorkin spaces (Theorem \ref{thm:local and global nonendpoint TL}). In Section \ref{dimension_one} we prove the optimal one dimensional results regarding the regularity of Fourier integral operators for all possible Banach and quasi-Banach scales. In Section \ref{section:sharpness} we give a motivation for why the boundedness results that we have obtained are sharp, and finally in Section \ref{section:applications in PDE} we produce the aforementioned local and global Besov-Lipschitz and Triebel-Lizorkin space estimates for hyperbolic partial differential equations (estimates \eqref{main global besov estimate for the wave equation} and \eqref{main global Triebel estimate for the wave equation} and Theorem \ref{local besov estimate for hyperbolic pde}).\\

 {\bf Acknowledgments:}   The  authors  are  grateful  to  the  referees, whose suggestions have improved the overall presentation of the paper. We are also indebted to Joachim Toft and Patrik Wahlberg for their comments and suggestions which have led to further improvements.

\section{Definitions and Preliminaries}\label{subsection:definitions}
\noindent In this section, we will collect all the definitions that will be used throughout this paper. We also state some useful results from both harmonic and microlocal analysis which will be used in the proofs of our results.\\

As is common practice, we will denote positive constants in the inequalities, which can be determined by known parameters in a given situation but whose
value is not crucial to the problem at hand, by $C$. Such parameters in this paper would be, for example, $m$, $p$, $q$, $s$, $n$,  and the constants connected to the seminorms of various amplitudes or phase functions. The value of $C$ may differ
from line to line, but in each instance could be estimated if necessary. We also write $a\lesssim b$ as shorthand for $a\leq Cb$ and moreover will use the notation $a\approx b$ if $a\lesssim b$ and $b\lesssim a$.\\

\noindent Let us recall the definition of the standard \textit{Littlewood-Paley decomposition} which  is a basic ingredient in our proofs, and is also used to define the function spaces that we are concerned with here.

\begin{defn}\label{def:LP}
Let $\psi_0 \in \mathcal C_c^\infty(\R^n)$ be equal to $1$ on $B(0,1)$ and have its support in $B(0,2)$. Then let
$$\psi_j(\xi) := \psi_0 \left (2^{-j}\xi \right )-\psi_0 \left (2^{-(j-1)}\xi \right ),$$
where $j\geq 1$ is an integer and $\psi(\xi) := \psi_1(\xi)$. Then $\psi_j(\xi) = \psi\left (2^{-(j-1)}\xi \right )$ and one has the following Littlewood-Paley partition of unity
\m{
\sum_{j=0}^\infty \psi_j(\xi) = 1 \quad \text{\emph{for all }}\xi\in\R^n.
}
\\

It is sometimes also useful to define a sequence of smooth and compactly supported functions $\Psi_j$ with $\Psi_j=1$ on the support of $\psi_j$ and $\Psi_j=0$ outside a slightly larger compact set. Explicitly, one could set
\m{
\Psi_j := \psi_{j+1}+\psi_j+\psi_{j-1},
}
with $\psi_{-1}:=\psi_0$.
\end{defn} \hspace*{1cm}\\

Using the Littlewood-Paley decomposition of Definition \ref{def:LP}, one can define the so called \textit{Besov-Lipschitz spaces} which are one of the main function spaces from the point of view of this paper.
\begin{defn}\label{def:Besov}
	Let $0 < p,q \le \infty$ and $s \in {\mathbb R}$. The Besov-Lipschitz spaces are defined by
	\[
	{B}^s_{p,q}(\R^n)
	:=
	\left\{
	f \in {\S}'(\R^n) \,:\,
	\|f\|_{{B}^s_{p,q}(\R^n)}
	:=
	\left(
	\sum_{j=0}^\infty
	2^{jq s}\|\psi_j(D)f\|^{q}_{L^p(\R^n)}
	\right)^{\frac{1}{q}}<\infty
	\right\}.
	\]
	
\noindent It is also worth mentioning that for $p=q=\infty$ and $0<s\leq 1$ we obtain the familiar Lipschitz space $\Lambda^s(\R^n)$, i.e. $B^s_{\infty,\infty}(\R^n)= \Lambda^s(\R^n)$.
\end{defn}

\begin{rem}
Different choices of the basis $(\psi_j)_{j=0}^\infty$ give equivalent \emph{(}quasi\emph{)}-norms of $B_{p,q}^s(\R^n)$ in \emph{Definition \ref{def:Besov},} see e.g. \cite{Trie83}. We will use either $(\psi_j)_{j=0}^\infty$ or  $(\Psi_j)_{j=0}^\infty$ to define the norm of $B_{p,q}^s(\R^n)$.
\end{rem}

We will also produce boundedness results in the realm of \textit{Triebel-Lizorkin} spaces which can be defined using Littlewood-Paley theory, as follows:

\begin{defn}\label{def:Triebel}
Let $0 < p,q \le \infty$ and $s \in {\mathbb R}$. The Triebel-Lizorkin spaces are defined by
	\[
	{F}^s_{p,q}(\R^n)
	:=
	\left\{
	f \in {\S}'(\R^n) \,:\,
	\|f\|_{{F}^s_{p,q}(\R^n)}
	:=
	\norm{\left(
	\sum_{j=0}^\infty
	2^{jq s}|\psi_j(D)f|^{q}
	\right)^{\frac{1}{q}}}_{L^p(\R^n)} <\infty
	\right\}.
	\]

\noindent Note that for $-\infty <s<\infty$ and $1\leq p<\infty,$ $F^s_{p,2}(\R^n)=H^{s,p}(\R^n)$ \emph{(}various $L^p$-based Sobolev and Sobolev-Slobodeckij spaces\emph{)} and for $0<p<\infty$, $F^0_{p,2}(\R^n)=h^p(\R^n)$ \emph{(}the local Hardy spaces\emph{)}. Moreover the dual space of $F^{0}_{1,2}(\R^n)$ is $\mathrm{bmo}$ \emph{(}the local version of $\mathrm{BMO}$\emph{)}.
\end{defn}

Another fact which will be useful to us is that for $-\infty <s<\infty$ and $0<p\leq \infty$
\begin{equation}\label{equality of TL and BL}
B^s_{p,p}(\R^n)= F^s_{p,p}(\R^n),
\end{equation}
and that one has the continuous embedding
\begin{equation}\label{embedding of TL}
F^{s+\varepsilon}_{p,q_0}(\R^n)\xhookrightarrow{} F^s_{p ,q_1}(\R^n),
\end{equation}
for $-\infty <s<\infty$, $0<p< \infty$, $0<q_0,q_1 \leq \infty$ and all $\varepsilon>0$.
Furthermore, for $s'\in \R$, the operator $\displaystyle \brkt{1-\Delta}^{\frac{s'}2}$ maps ${F}^s_{p,q}(\R^n)$ isomorphically into ${F}^{s-s'}_{p,q}(\R^n)$ and ${B}^s_{p,q}(\R^n)$ isomorphically into ${B}^{s-s'}_{p,q}(\R^n).$\\

Since we shall later on specifically deal with Triebel-Lizorkin spaces \linebreak$F_{p,2}^0(\R^n)=h^p (\R^n)$, we also recall that a function $a$ is called a $h^p$-atom if for some $x_0\in \R^n$ and $r>0$ the following three conditions are satisfied:
\begin{enumerate}
\item $\supp a\subset B(x_{0}, r)$,
\item $\displaystyle |a(x)|\leq|B(x_{0}, r)|^{-\frac 1p},$
\item If $r\leq 1$, $\displaystyle M\geq \left[ n\brkt{\frac 1p -1}_+ \right ]$, where $[x]$ denotes the integer part of $x$, then $\displaystyle \int_{\R^n} x^{\alpha}a(x)\dd x=0,$ for $|\alpha|\leq M$. No further condition is assumed if $r>1.$ 
\end{enumerate}

It is well known (see \cite{Trie83}) that a distribution $f\in h^p (\R^n)$ has an atomic decomposition
$$
f=\sum_{j}\lambda_{j}a_{j},
$$
where the $\lambda_{j}$ are constants such that $$\displaystyle \sum_{j}|\lambda_{j}|^{p}\approx\Vert f\Vert_{h^p(\mathbb{R}^{n})}^{p}=\Vert f\Vert_{F_{p,2}^0(\mathbb{R}^{n})}^{p}$$ and the $a_{j}$ are $h^p$-atoms.\\

Another important and useful fact about Besov-Lipschitz and Triebel-Lizorkin spaces is the following:

\begin{thm}\label{invariance thm}
Let $\eta: \R^n \to\R^n$ with $\eta(x)=(\eta_1 (x), \dots,\eta_n (x))$ be a diffeomorphism such that $\abs{\det D\eta (x)}\geq c>0$, $\forall x\in \R^n$ \emph{(}$D\eta$ denotes the Jacobian matrix of $\eta$\emph{)}, and $\Vert\partial^{\alpha}\eta_j (x)\Vert_{L^\infty(\R^n)}\lesssim 1$ for all $j\in \{1,\dots, n\}$ and $|\alpha|\geq 1.$ Then
for $s\in \R$, $0<p<\infty$ and $0<q\leq \infty$ one has $$\Vert f\circ \eta\Vert_{F^{s}_{p,q}(\R^n)}\lesssim \Vert f\Vert_{F^{s}_{p,q}(\R^n)}.$$

The same invariance estimate is also true for Besov-Lipschitz spaces $B^{s}_{p,q}(\R^n)$ for  $s\in \R$, $0<p\leq\infty$ and $0<q\leq \infty$.
\end{thm}
For a proof see J. Johnsen, S. Munch Hansen and W. Sickel \cite{JohnsenMunchSickel}*{Corollary 25}, and  H. Triebel \cite{Triebel 2}*{Theorem 4.3.2}. References \cite{Trie83} and \cite{Triebel 2} and \cite{Triebel 3} are actually the standard references for all the facts concerning Besov-Lipschitz and Triebel-Lizorkin spaces. See also \cite{Triebelpseudo} for a summary of most important properties of the Triebel-Lizorkin spaces. \\

\noindent Next we recall the definition of two classes of amplitudes which are the basic building blocks of the pseudodifferential and the Fourier integral operators used in this paper.
The first class was first introduced by J.J. Kohn and L. Nirenberg in \cite{KN}.
\begin{defn}\label{symbol class Sm}
An \textit{amplitude} \emph{(}symbol\emph{)} $a(x,\xi)$ in the class $S^m(\R^n)$ is a function \\$a\in \mathcal{C}^\infty (\R ^n\times \R ^n)$ that verifies the estimate
\m{
\left |\partial_\xi^\alpha \partial_x^\beta a(x,\xi) \right |\lesssim \la\xi\ra ^{m-|\alpha|},
}
for all multi-indices $\alpha$ and $\beta$ and $(x,\xi)\in \R ^n\times \R ^n$, where $\la\xi\ra:= \brkt{1+|\xi|^2}^{\frac{1}{2}}.$
We shall henceforth refer to $m$ as the order of the amplitude.
We shall also use the class of amplitudes $A^m(\R^n)$ which consists of all $a(x,y,\xi)$ that verify the estimate
\m{
| \partial_\xi^\alpha \partial_x ^\beta  \partial_y ^\gamma  a(x, y,\xi) |\lesssim \la\xi\ra ^{m-|\alpha|},
}
for all multi-indices $\alpha$, $\beta$, $\gamma$ and $(x, y,\xi)\in \R ^n \times \R ^n \times \R ^n$.
\end{defn}
\noindent There is another class of amplitudes used in Proposition \ref{prop:JFA} below, that are those which are merely bounded in the $x$-variable and
were first introduced by  C. Kenig and W. Staubach in \cite{KS}.
\begin{defn}\label{symbolclass LinfSm}
\noindent An \textit{amplitude} \emph{(}symbol\emph{)} $a(x,\xi)$ is in the class $L^\infty S^m(\R^n)$  if it is essentially bounded in the $x$ variable, $\mathcal{C}^\infty ( \R ^n)$ in the $\xi$ variable and verifies the estimate
\m{
\norm{ \partial_\xi^\alpha a(\cdot,\xi) }_{L^\infty(\R^n)}\lesssim \la\xi\ra ^{m-|\alpha|},
}
for all multi-indices $\alpha$ and $\xi\in \R ^n$.
\end{defn}
We note that $S^m (\R^n)\subset L^\infty S^m(\R^n).$\\

\noindent For the purpose of proving boundedness results for Fourier integral operators, it turns out that the following order of the amplitude is the critical one, namely
\nm{eq:criticaldecay}{m_c(p) :=  -(n-1)\abs {\frac 1p -\frac 12},}
where $0<p\leq\infty$.
This means that, we will be able to establish various boundedness results for the Fourier integral operators when the order of the amplitude is less than or equal to $m_c(p)$.\\

\noindent Given the symbol classes defined above, one associates to the symbol its \textit{Kohn-Nirenberg quantisation }as follows:
\begin{defn}
Let $a$ be a symbol. Define a pseudodifferential operator \emph{(}$\Psi\mathrm{DO}$ for short\emph{)} as the operator
\begin{equation*}
a(x,D)f(x) := \int_{\R^n}e^{ix\cdot\xi}a(x,\xi)\widehat{f}(\xi) \ddd \xi,
\end{equation*}
a priori defined on the Schwartz class $\mathscr{S}(\R^n).$ Here and in what follows, \linebreak$\displaystyle \ddd \xi:= (2\pi)^{-n}\dd \xi.$
\end{defn}\hspace*{1cm}\\
\noindent In order the define the Fourier integral operators that are studied in this paper, we also define the classes of phase functions, which together with the amplitudes of Definitions \ref{symbol class Sm} and \ref{symbolclass LinfSm} are the main building blocks of Fourier integral operators.

\begin{defn}\label{def:phi2}
A \textit{phase function} $\varphi(x,\xi)$ in the class $\Phi^k$ is a function \linebreak$\varphi(x,\xi)\in \mathcal{C}^{\infty}(\R^n \times\R^n \setminus\{0\})$, positively homogeneous of degree $1$ in the frequency variable $\xi$ satisfying the following estimate

\begin{equation}\label{C_alpha}
	\sup_{(x,\,\xi) \in \R^n \times\R^n \setminus\{0\}}  |\xi| ^{-1+\vert \alpha\vert}\left | \partial_{\xi}^{\alpha}\partial_{x}^{\beta}\varphi(x,\xi)\right |
	\leq C_{\alpha , \beta},
	\end{equation}
	for any pair of multi-indices $\alpha$ and $\beta$, satisfying $|\alpha|+|\beta|\geq k.$
    In this paper we will mainly use  phases in class $\Phi^2$ and occasionally  also $\Phi^1$.
\end{defn}\hspace*{1cm}

\noindent We will also need to consider phase functions that satisfy  certain {\em non-degeneracy conditions}. These conditions have to be adapted to the case of local and global boundedness in an appropriate way. Following \cite{SSS}, in connection to the investigation of the local results, that is, under the assumption that the $x$ support of the amplitude $a(x,\xi)$ lies within a fixed compact set $\mathcal{K}$, the non-degeneracy condition is formulated as follows:
\begin{defn}
Let $\mathcal{K}$ be a fixed compact subset of $\R^n$. One says that the phase function $\varphi(x,\xi) $  satisfies the non-degeneracy condition if
\begin{equation}\label{ND}
	\det \brkt{\partial^{2}_{x_{j}\xi_{k}}\varphi(x,\xi)}\neq 0,\qquad \mbox{for all $(x,\xi)\in \mathcal{K}\times \R^n\setminus\{0\}$}.
\end{equation}
\end{defn}\hspace*{1cm}\\
\noindent Following the approach in e.g. \cite{JFA}, for the global $L^p$ boundedness  results that were established in that paper, we also define the following somewhat stronger notion of non-degeneracy:
\begin{defn}
One says that the phase function $\varphi(x,\xi)$ satisfies the strong non-degeneracy condition \emph{(}or $\varphi$ is $\mathrm{SND}$ for short\emph{)} if
\begin{equation}\label{SND}
	\left |\det \brkt{\partial^{2}_{x_{j}\xi_{k}}\varphi(x,\xi)} \right |\geq \delta,\qquad \mbox{for  some $\delta>0$ and all $(x,\xi)\in \mathbb{R}^{n} \times \R^n\setminus\{0\}$}.
\end{equation}
\end{defn}\hspace*{1cm}\\

We define a \lq\lq influence set" of the SND phase function $\varphi$.
\begin{defn}\label{def:influenceset}
Let $\bar y\in\R^n$ be the centre of a ball $B$ with radius $r$. We define the \lq\lq rectangles" $R_{j}^{\nu}$ by
$$
R_{j}^{\nu}=\set{x\in\ \R^n\ :\ |\nabla_{\xi}\varPhi(x,\overline{y},\xi_{j}^{\nu})|\leq A 2^{-\frac j2},\ |\pi_{j}^{\nu}(\nabla_{\xi}\varPhi(x,\overline{y},\xi_{j}^{\nu}))|\leq A 2^{-j}}.
$$
where $\pi_j^\nu$ is the orthogonal projection in the direction $\xi_{j}^{\nu}$ and $\varPhi$ is of either the form  $\varPhi(x,y,\xi)= \varphi(x,\xi) -y\cdot\xi$ or $\varPhi(x,y,\xi)=x\cdot\xi -\varphi(y,\xi)$. The size of the constant $A$ depends on the size of various Hessians of $\Phi$ but not on $j$.
One then defines the \lq\lq influence set"
\begin{equation}\label{defn:Bstar}
B^{*}=\bigcup_{2^{-j}\leq r}\bigcup_{\nu}R_{j}^{\nu}.
\end{equation}

\end{defn}
\begin{rem}\label{rem:influenceset}
Given $B^*$ in \emph{Definition \ref{def:influenceset}} above, one can show the following:\\
\begin{enumerate}
\item[\emph{(i)}]$|B^{*}|\lesssim r$ \emph{(}see e.g. \cite{Stein}\emph{)}. \\
\item[\emph{(ii)}] Suppose that $k$ is an integer such that $2^{-k}\leq r < 2^{-k+1}$, that $x\in B^{*c}$ and that $y\in B$. Then there is a unit vector $\xi_k^\mu$ such that $\abs{\xi_j^\nu-\xi_k^\mu}\leq 2^{-\frac k2}$ and, by using homogeneity and the triangle inequality, there exists a constant $C>0$ such that
\begin{equation*}\label{raeddaren}
\big|2^{j}\nabla_{\xi_{1}}\varPhi (x,y,\xi_{j}^{\nu}))\big|+\big|2^{\frac j2}\nabla_{\xi'}\varPhi (x,y,\xi_{j}^{\nu})\big|\geq C2^{\frac{j-k}2}.
\end{equation*}
\end{enumerate}
\end{rem}\hspace*{1cm}\\

\noindent Having the definitions of the amplitudes and the phase functions at hand, one has
\begin{defn}\label{def:FIO}
	A Fourier integral operator \emph($\mathrm{FIO}$ for short\emph) $T_a^\varphi$ with amplitude $a$ and phase function $\varphi$ satisfying \eqref{ND}, is defined \emph{(}once again a-priori on $\mathscr{S}(\R^n)$\emph{)} by formula \eqref{definition of FIO} in the introduction.
\end{defn}\hspace*{1cm}\\
\noindent The following composition result, whose proof can be found in \cite{MONSTERIOSITY}*{Theorem 4.2} will enable us to keep track of the parameter while a parameter-dependent $\Psi$DO is composed with a parameter-dependent FIO. This will be crucial in some of the forthcoming proofs.
\begin{thm}\label{prop:monsteriosity}
Let $m \leq 0$, $\displaystyle 0<\varepsilon <\frac 12$ and $\Omega := \R^n \times \{|\xi| > 1\}$. Suppose that $ a_t(x, \xi)\in S^m (\R^n)$ uniformly in $t \in (0, 1]$ and it is supported in $\Omega$, $\rho(\xi)\in S^0(\R^n)$ and $\varphi \in \mathcal C^\infty (\Omega)$ is such that
\begin{enumerate}
\item[\emph{(i)}]for constants $C_1, C_2 > 0$, $C_1|\xi| \leq |\nabla_x \varphi(x, \xi)| \leq C_2|\xi|$ for all $(x, \xi) \in \Omega$, and
\item[\emph{(ii)}]for all $|\alpha|, |\beta| \geq 1$, $|\partial_x^\alpha \varphi(x, \xi)|\lesssim  \la \xi \ra$ and $|\partial_\xi^\alpha \partial _x^\beta \varphi (x, \xi)| \lesssim  |\xi|^{1-|\alpha|}$, for all $(x, \xi) \in \Omega$.
\end{enumerate}

\noindent Consider the parameter dependent Fourier integral operator $T_{a_t}^\varphi$, given by \eqref{definition of FIO} with amplitude $a_t(x,\xi)$, and the parameter dependent Fourier multiplier  
\begin{equation*}
\rho(tD)f(x) := \int_{\R^n} e^{ix\cdot \xi}\,\rho(t\xi)\,\widehat f(\xi) \ddd \xi
\end{equation*}
and let $\sigma_t$ be the amplitude of the composition operator $\rho (tD)T_{a_t}^\varphi = T_{\sigma_t}^\varphi$ given by
\begin{equation*}
\sigma_t(x, \xi) := \iint_{\R^{n}\times \R^n} a_t(y, \xi)\,\rho (t\eta)\,e^{i(x-y)\cdot \eta+i\varphi(y,\xi)-i\varphi(x,\xi)} \ddd\eta \dd y.
\end{equation*}
Then, for each $M \geq 1$, we can write $\sigma_t$ as
\begin{equation*}
\sigma_t(x, \xi) = \rho (t \nabla_x \varphi(x, \xi))\,a_t(x, \xi) + \sum_{0<|\alpha|<M} \frac{t^{|\alpha|}}{\alpha!}
 \sigma_\alpha (t, x, \xi) + t^{M\varepsilon} r(t, x, \xi)
\end{equation*}
for $t \in (0, 1)$. Moreover, for all multi-indices $\beta, \gamma$ one has
\begin{equation*}
\sup_{t\in(0,1)} \abs{\partial ^\gamma_\xi \partial_x^\beta \sigma_\alpha(t, x, \xi)   t^{|\alpha|(1-\varepsilon)}} \lesssim \la\xi\ra^{m-|\alpha|\left ( \frac 12- \varepsilon \right )-|\gamma|}\text{ for } 0 < |\alpha | < M,
\end{equation*}
and
\begin{equation*}
\sup_{t\in(0,1)} \abs{\partial _\xi^\gamma \partial_x^\beta r(t, x, \xi)} \lesssim   \la\xi\ra ^{m-M\left ( \frac 12 -\varepsilon \right )-|\gamma|}.
\end{equation*}
\end{thm}\hspace*{1cm}\\

To deal with the low frequency portion of the kernels of FIOs, which are frequency supported in a neighbourhood of the origin (where the phase function is singular), the following lemma which was proven in \cite{memoirs}*{Lemma 1.17}, will come in handy.
\begin{lem}\label{lem:David-Wulf}
Let $b(x, \xi)$ be a bounded function which is compactly supported in the $\xi$ variable and also belongs to $\mathcal C^{n+1}(\R^n  \setminus \{0\})$ in $\xi$. Moreover assume that $b(x,\xi)$ satisfies
\m{
\sup_{\xi\in \R^n \setminus \{0\}}\abs{\xi}^{-1+|\alpha|}\norm{\partial^\alpha_\xi b(\cdot,\xi)}_{L^\infty(\R^n)} <\infty,
}
for $|\alpha|=n+1.$ Then for all $\mu\in [0,1)$
\m{
\sup_{x,y\in \R^n} \la y\ra^{n+\mu} \abs{\int _{\R^n} e^{-iy\cdot \xi}\ b(x,\xi)\ddd \xi}  <\infty.
}
\end{lem}\hspace*{1cm}\\

The following phase reduction lemma, whose proof can be found in \cite{memoirs}*{Lemma 1.10}, will reduce the phase of the Fourier integral operators to a linear term plus a phase for which the first order frequency derivatives are bounded.
\begin{lem}\label{phase reduction lem}
Any Fourier integral operator $T_\sigma ^\varphi $ of the type \eqref{definition of FIO} with amplitude $\sigma(x,\xi)\in S^{m}(\R^n)$ and phase function $\varphi(x,\xi)\in\Phi^2$, can be written as a finite sum of operators of the form
\m{
   \int_{\R^n} a(x,\xi)\, e^{i\theta(x,\xi)+i\nabla_{\xi}\varphi(x,\zeta)\cdot\xi}\, \widehat{u}(\xi)  \ddd\xi
}
where $\zeta$ is a point on the unit sphere $\mathbb{S}^{n-1}$, $\theta(x,\xi)\in \Phi^{1},$ and $a(x,\xi)$ is localised in the $\xi$ variable around the point $\zeta$.
Moreover, if one has a Fourier integral operator of the form $$\displaystyle \iint_{\R^n\times \R^n} a(y,\xi)\, e^{i\varphi(y,\xi)-ix\cdot\xi}\, u(y) \ddd \xi \dd y,$$with $\varphi\in \Phi^2$ then this operator can be written as a finite sum of operators
\m{
   \iint_{\R^n\times \R^n} a(y,\xi)\, e^{i\theta(y,\xi)+i\nabla_{\xi}\varphi(y,\zeta)\cdot\xi-ix\cdot\xi}\, u(y)  \ddd\xi \dd y,
}
where $\theta(y,\xi)\in \Phi^{1},$  and $a(y,\xi)$ is localised in the $\xi$ variable around the point $\zeta$.
\end{lem}

We will state the following lemma originally due to J. Peetre \cite{Peetre}, whose proof can be found in \cite{Trie83}*{Section 2.3.6}, which in combination with the previous lemma, turns out to be very useful later on in proving the boundedness of the low frequency part of FIOs. 
\begin{lem}\label{grafakos lemma 1}
Let $f\in \mathcal C^1(\R^n)$ with Fourier support inside the unit ball. Then for every $\rho>0$ and $\displaystyle r\geq\frac{n}{\rho}$ one has
\m{
\left (\langle \cdot\rangle^{-\rho} \ast |f|\right )(x)\lesssim \Big (M(|f|^r)(x)\Big ) ^{\frac{1}{r}},
}
where $M$ denotes the Hardy-Littlewood maximal function on $\R^n$.
\end{lem}\hspace*{1cm}\\

\noindent Since pseudodifferential operators are not in general $L^p $ bounded for $0<p\leq 1$, we will also need a weaker version of an $L^p$ space. Hence, following H. Triebel \cite{Trie83}, we define the $L^p$ spaces with compact Fourier support.
\begin{defn}\label{def:Lp Fourier support}
Let $0<p\leq \infty$ and $\mathcal{K} \subset \R^n$ be a compact set. Define
\begin{equation*}
L_\mathcal{K}^p(\R^n) := \left \{ f\in \S ' (\R^n) \, :\, \norm{f}_{L^p(\R^n)} <\infty,  \,\supp \widehat f \subset \mathcal{K}  \right \}
\end{equation*}
Observe that other authors may use the notation $L_p^\mathcal{K}(\R^n)$, see e.g. \cite{Trie83}.
\end{defn}
\noindent In connection to this and the convolution of functions in $L^p_\mathcal{K}(\R^n)$ spaces, the following lemma, whose proof can be found in Remark 2 of \cite{Trie83}*{p. 28}, is quite useful.
\begin{lem}\label{lem:hassesfaltning}
Let $\mathcal{K} := \overline{B(0,r)}$ for some $r>0$ and let $f,g\in L^p_\mathcal{K}(\R^n)$ for $0<p\leq 1$. Then
\m{
\norm{f*g}_{L^p(\R^n)}\lesssim r^{n\left ( \frac 1p -1\right )} \norm{f}_{L^p(\R^n)}\norm{g}_{L^p(\R^n)}.
}

\end{lem}
In establishing the local boundedness of FIOs for the optimal ranges of $p$'s, the following Bernstein-type estimate will be useful. The proof can be found in \linebreak\cite{Trie83}*{p. 22}.

\begin{lem}\label{lem:bernstein}
Let $\mathcal{K} \subset \R^n$ be a compact set and let $0<p\leq r \leq \infty$. Then
\m{
\norm {\partial^\alpha f }_{L^r(\R^n)} \lesssim \norm{f}_{L^p(\R^n)}
}
for all multi-indices $\alpha$ and all $f\in L^p_\mathcal{K} (R^n)$.
\end{lem}\hspace*{1cm}\\

\noindent In order to establish $L^p$ estimates ($0< p \leq1$) for a generic Littlewood-Paley piece of the FIOs, we will estimate the so called \textit{Peetre's maximal functions} by the Hardy-Littlewood maximal operators as in the following lemma:
\begin{lem}\label{lem:Triebel-Schmeisser}
Let $\big \{f_{j,k} \big \}\in L^p_{\mathcal K_{j,k}}(\R^n)$ with $$\mathcal K_{j,k}:= \left \{ (\xi_1,\xi_2)\in \R^{n'} \times \R^{n-n'}:\ |\xi_1|\leq c_{j,k}, |\xi_2|\leq d_{j,k} \right \}$$ where $0\leq n'\leq n$ and $c_{j,k}$ and $d_{j,k}$ are some positive constants. Futhermore, let $x:=(x_1,x_2)\in \R^{n'} \times \R^{n-n'}$ and $z:=(z_1,z_2) \in \R^{n'} \times \R^{n-n'}$. Then one has
\begin{equation*}
\sup_{z\in \R^n}{\frac{|f_{j,k} (z)|}{\left (1+(c_{j,k}|x_1-z_1|)^{\frac 1 {r_1}} \right ) \left (1+(d_{j,k} |x_2-z_2|)^{\frac 1 {r_2}}\right )}} \lesssim \left (M_{2}\left (M_{1}\left |f_{j,k} \right |^{r_1} \right )^{\frac{r_2}{r_1}}\right )^{\frac{1}{r_2}}(x)
\end{equation*}
uniformly in $j,k$ and for all $r_1,r_2>0$ small enough. Here $M_{1}$ is the Hardy-Littlewood maximal function acting on the function in the $x_1$ variable, i.e.
\begin{equation*}
M_{1}f(x):= \sup_{\delta>0} \frac{1}{|B(x_1,\delta)|} \int_{B(x_1,\delta)} |f(y_1,x_2)| \dd y_1,
\end{equation*}
and $M_{2}$ is defined in a similar way.
\end{lem}
\begin{proof}
Let $g\in L_{\mathcal K}^p(\R^n)$ for $\mathcal K :=  \left \{ (\xi_1,\xi_2)\in \R^{n'} \times \R^{n-n'}:\ |\xi_1|\leq 1, |\xi_2|\leq 1 \right \}$. Then
\begin{equation}\label{eq:triebel-schmeisser}
\sup_{z\in \R^n}{\frac{|\partial^\alpha g (z)|}{\left (1+|x_1-z_1|^{\frac 1 {r_1}} \right ) \left (1+ |x_2-z_2|^{\frac 1 {r_2}}\right )}} \lesssim \left (M_{2}\big (M_{1}\left |g \right |^{r_1} \big )^{\frac{r_2}{r_1}}\right )^{\frac{1}{r_2}}(x).
\end{equation}
The proof of \eqref{eq:triebel-schmeisser} when $n=2$ and $n'=1$ can be found in \cite{Triebel-Schmeisser}*{p. 48, equation (1)}. By carefully tracing that proof, one can generalise the result to \eqref{eq:triebel-schmeisser}. Then the lemma follows by setting $f_{j,k}(x_1,x_2) := g(c_{j,k}x_1,d_{j,k}x_2)$ and $\alpha=0$.
\end{proof}\hspace*{1cm}\\

\noindent Finally we state the following version of the non-stationary phase lemma, whose proof can be found in \cite{JFA}*{Lemma 3.2}.
\begin{lem}\label{technic}
Let $\mathcal{K}\subset \mathbb{R}^n$ be a compact set and  $\Omega \supset \mathcal{K}$ an open set. Assume that $\Phi$ is a real valued function in $\mathcal C^{\infty}(\Omega )$ such that $|\nabla \Phi|>0$ and
$|\d^\alpha \Phi| \lesssim |\nabla \Phi|$
for all multi-indices $\alpha$ with $|\alpha|\geq 1$.
Then, for any $F\in \mathcal C^\infty_c (\mathcal{K})$, $\lambda>0$ and any integer $k\geq 0$,
\begin{equation*}
	 \abs{ \int_{\R ^n} F(\xi)\, e^{i\lambda \Phi(\xi)}\, \ddd \xi }
     \leq C_{k,n,\mathcal{K}} \lambda^{-k} \sum_{|\alpha| \leq k} \int_\mathcal{K} |\d^{\alpha} F(\xi)| \, |\nabla \Phi(\xi)|^{-k}\, \ddd \xi.
\end{equation*}
\end{lem}\hspace*{1cm}\\

\section{The Seeger-Sogge-Stein decomposition and the associated kernel estimates}\label{SSS section}
In connection to the study of the $L^p$ regularity of FIOs, A. Seeger. C. Sogge and E. Stein introduced a second dyadic decomposition superimposed on a preliminary Littlewood-Paley decomposition, in which each dyadic shell $\left \{2^{j-1}\leq
\vert \xi\vert\leq 2^{j+1}\right \}$ (as in Definition \ref{def:LP}) is further partitioned into truncated cones of thickness roughly $2^{\frac{j}{2}}$ and one can prove that $O\brkt{2^{j\frac{n-1}{2}}}$
such elements are needed to cover one shell.
\begin{defn}\label{def:LP2}
For each $j\in\N$ we fix a collection of unit
vectors $\big \{\xi^{\nu}_{j}\big \} $ that satisfy the following two conditions.
\begin{enumerate}
\item [\emph{(i)}]$\displaystyle \big | \xi^{\nu}_{j}-\xi^{\nu'}_{j} \big |\geq 2^{-\frac{j}{2}},$ if $\nu\neq \nu '$.
\item [\emph{(ii)}]If $\xi\in\mathbb{S}^{n-1}$, then there exists a $
\xi^{\nu}_{j}$ so that $\big \vert \xi -\xi^{\nu}_{j} \big \vert
<2^{-\frac{j}{2}}$.
\end{enumerate}
One can take a collection $\{\xi_{j}^{\nu}\}$ which is maximal with respect to the first property and there are at most $O\brkt{2^{j\frac {n-1}2}}$ elements in the collection $\{\xi_{j}^{\nu}\}$.\\ 

Let $\Gamma^{\nu}_{j}$ denote the cone in the $\xi$ space whose
 central direction is $\xi^{\nu}_{j}$, i.e.
\begin{equation}\label{eq:gammajnu}
    \Gamma^{\nu}_{j}:=\set{ \xi\in\R^n :\,  \abs{ \frac{\xi}{\vert\xi\vert}-\xi^{\nu}_{j}}\leq 2\cdot 2^{-\frac{j}{2}}}.
\end{equation}
One also defines 
\m{\eta_j^\nu (\xi) := \phi\brkt{2^{\frac j2}\brkt{\frac {\xi}{|\xi|}-\xi_j^\nu}},}
where $\phi$ is a nonnegative function in $ \mathcal C_c^\infty(\R^n)$ with $\phi(u)=1$ for $|u|\leq 1$ and $\phi(u)=0$ for $u\geq 2$.
\end{defn}

\noindent As was done in \cite{SSS} one could set $$\chi_j^\nu := \eta_j^\nu \brkt{\sum_{\nu} \eta_j^\nu}^{-1}$$ which is in $\mathcal C^\infty(\R^n\setminus \{0\} )$ and supported in the cone $\Gamma_j^\nu$ satisfying the estimates \nm{eq:quadrseconddyadcond}{\abs{\d_\xi^\alpha \chi_j^\nu (\xi) }\lesssim 2^{j\frac {\abs \alpha}2}\abs\xi^{-\abs\alpha}} for all multi-indices $\alpha$ and
\nm{eq:quadrseconddyadcond2}{ \left \vert \partial
^{N}_{\xi_{1}}\chi^{\nu}_{j}(\xi) \right \vert\leq C_{N}
\vert \xi\vert ^{-N}, \quad  \text{for} \quad N\geq 1,}
if one chooses the axis in $\xi$-space such that $\xi_1$ is in the direction of $\xi^{\nu}_{j}$ and \linebreak $\xi'=(0,\xi _2 , \dots , \xi_{n})$ is perpendicular to $\xi^{\nu}_{j}$. With this construction, it is also clear that
\begin{equation}\label{the first partition of unity}
\sum _{\nu}\chi^{\nu}_{j}(\xi) =1,\quad \text{ for all } j\text { and } \xi\neq 0.
\end{equation}

Therefore, if $\psi_j$ is chosen as in Definition \ref{def:LP}, one has \begin{equation}\label{PL partition of unity}
\psi_{0}(\xi)+\sum_{j=1}^{\infty}\sum_{\nu}\chi_{j}^{\nu}(\xi)\psi_{j}(\xi)=1, \quad \text{for all } \xi\in \mathbb{R}^{n}.
\end{equation}

It is sometimes useful to use a slightly different partition of unity by setting \begin{equation}\label{chitilde}
\tilde{\chi}_j^\nu := \eta_j^\nu \brkt{\sum_\nu \big(\eta_j^\nu\big)^2}^{-\frac{1}{2}},
\end{equation}
which satisfies \nm{eq:quadrseconddyad}{\sum_\nu \tilde{\chi}_j^\nu(\xi)^2 = 1,\quad \text{ for all } j\text { and } \xi\neq 0.} Once again, one can show that \eqref{eq:quadrseconddyadcond} and \eqref{eq:quadrseconddyadcond2} are also satisfied for $\tilde{\chi}_j^\nu.$\\

 Using the Littlewood-Paley localisation $\psi_j$ and the second dyadic frequency localisation $\chi_j^\nu$, we have the following estimate for the localised high frequency part of the kernels:

\begin{lem}\label{lem:kernelestimate}
Let $j\geq 1$ and set \nm{eq:SSS kernel defintion}{K_j^\nu (x,y) := \int_{\R^n} e^{i\varPhi(x,y,\xi)}\psi_j(\xi)\chi_j^\nu(\xi) a(x,y,\xi)\ddd\xi } where  $a\in A^m(\R^n),$
$\varPhi(x,y,\xi)= \varphi(x,\xi) -y\cdot\xi$ or $\varPhi(x,y,\xi)=x\cdot\xi -\varphi(y,\xi)$
and $\varphi(x,\xi)\in \Phi^2.$
Then for all $N\geq 0$, the kernel $K_j^\nu$ satisfies the estimate 
\nm{eq:claimkernel}
{
|{\partial_y^\alpha}K_{j}^{\nu}(x, y)|  \lesssim \frac{2^{j\brkt{m+\frac{n+1}{2}+\abs\alpha}}}{\left  (1+\big|2^{j}\nabla_{\xi_1}\varPhi(x,y,\xi_{j}^{\nu})\big|^{2}\right )^{N} \left (1+\big |2^{\frac j2}\nabla_{\xi'}\varPhi(x,y,\xi_{j}^{\nu})\big |^{2}\right )^{N}}.
}
\end{lem}

\begin{proof}
It is enough to show the case for $\alpha=0$ since $\alpha$ derivatives in the $y$-variable introduce factors bounded by $2^{j\abs\alpha}$. Define $ h(x,y,\xi):=\varPhi(x,y,\xi)-\nabla_{\xi}\varPhi(x,y,\xi_{j}^{\nu})\cdot\xi$. Then one has
\begin{equation*}
\begin{split}
K_{j}^{\nu}(x,y)=\int_{\mathbb{R}^{n}} e^{i\nabla_{\xi}\varPhi(x,y,\xi_{j}^{\nu})\cdot\xi}\, b_{j}^{\nu}(x,y,\xi)\ddd\xi,
\end{split}
\end{equation*}
where $b_{j}^{\nu}(x,y, \xi):=\psi_j(\xi)\chi_{j}^{\nu}(\xi)e^{ih(x,y,\xi)}$. It can be verified (see e.g. \cite{Stein}*{p. 407}) that the phase
$h(x,y,\xi)$ satisfies
\begin{equation}\label{phaseestim1}
\left \vert \partial^{N}_{\xi_{1}}
h(x,y,\xi) \right \vert\leq C_{N} 2^{-jN}
\end{equation}
\begin{equation}\label{phaseestim2}
 \; \; \big | \partial_{\xi'}^{\alpha'}h(x,y,\xi)  \big | \leq
C_{N} 2^{-j\frac{\abs{\alpha'}}{2}},
\end{equation}
for $N\geq 2$ on the support of $b^{\nu}_{j}(x,y,\xi)$. Introducing the differential operator \\$L:=\left (I-2^{2j}\partial_{\xi_{1}}^{2}\right ) \Big (I-2^{j}\Delta_{\xi '} \Big)$, one has
\begin{equation*}
  \begin{split}
  e^{i\nabla_{\xi}\varPhi(x,y,\xi_{j}^{\nu})\cdot\xi} =\frac{L^{N}e^{i\nabla_{\xi}\varPhi(x,y,\xi_{j}^{\nu})\cdot\xi}}{\left (1+\big|2^{j}\nabla_{\xi_{1}}\varPhi(x,y,\xi_{j}^{\nu})\big|^{2}\right )^{N}  \left (1+\big|2^{\frac j2}\nabla_{\xi'}\varPhi(x,y,\xi_{j}^{\nu})\big|^{2}\right )^{N}} ,
  \end{split}
\end{equation*}
Furthermore for $b_{j}^{\nu}(x, y,\xi)$ using the assumption that $a\in S^{m}(\R^n)$ together with (\ref{eq:quadrseconddyadcond}), (\ref{eq:quadrseconddyadcond2}), and the uniform estimates (in $x$) for $h(x,y,\xi)$ in \eqref{phaseestim1} and (\ref{phaseestim2}), we can show that for any $j\in \N$, $\nu$ and $\xi\in \supp_\xi {b^\nu_j}$
\begin{equation}\label{eq:b-boundedness}
\left \Vert L^N b^{\nu}_{j}(\cdot,\cdot,\xi) \right \Vert_{L^\infty(\R^n\times \R^n)}\leq C_{N} 2^{jm}.
\end{equation}
Now integration by parts yields
\begin{equation*}
  \begin{split}
    \left |K_{j}^{\nu}(x, y) \right | &\lesssim \frac{\displaystyle\int_{\supp_\xi \,b_{j}^{\nu}}\left |L^{N}b_{j}^{\nu}(x,y,\xi)\right |\ddd\xi}{\left (1+\big|2^{j}\nabla_{\xi_{1}}\varPhi(x,y,\xi_{j}^{\nu})\big|^{2}\right )^{N}\left (1+\big|2^{\frac j2}\nabla_{\xi'}\varPhi(x,y,\xi_{j}^{\nu})\big|^{2}\right )^{N}}  \\  &\lesssim  \, \frac{2^{jm}2^{j\frac{n+1}{2}}}{\left (1+\big|2^{j}\nabla_{\xi_{1}}\varPhi(x,y,\xi_{j}^{\nu})\big|^{2}\right )^{N} \left (1+\big|2^{\frac j2}\nabla_{\xi'}\varPhi(x,y,\xi_{j}^{\nu})\big|^{2}\right )^{N}}
  \end{split}
\end{equation*}
where we used (\ref{eq:b-boundedness}) and that $\left |\supp b_j^\nu \right |=O\brkt{2^{j\frac{n+1}{2}}}$. Hence the proof is complete.
\end{proof}
\begin{rem}
The conclusion of \emph{Lemma \ref{lem:kernelestimate}}, is also valid if the phase function $\varphi$ is merely assumed to be in $\mathcal{C}^{\infty}(\R^n \times\R^n \setminus\{0\})$ and positively homogeneous of degree one in $\xi$.
\end{rem}\hspace*{1cm}\\
We now prove the following lemma, which is used for the estimates of the operator $T_\sigma^\varphi$ in the proof of Proposition \ref{Thm:high TL}.
\begin{lem}\label{SSS estimates for Kj}
For $j\geq 1$  and $0<p\leq 1$, let
\m{K_{j}^{\nu}(x, y):=\int_{\mathbb{R}^{n}}\sigma(x,y,\xi)\,\psi_j(\xi)\chi_{j}^{\nu}(\xi)\, e^{i\varphi(x,\xi) -iy\cdot\xi} \ddd\xi }
and
\m{T_j^\nu a(x) = \int_{\R^n}K_j(x,y)a(y)\dd y}
where  $\sigma \in A^{m}(\R^n)$, an $h^p$ atom $a$ supported in the ball $B(\overline{y}, r)$ and $\varphi\in \Phi^2 $ satisfies the \emph{SND} condition \eqref{SND}. Moreover let ${B^*}$ be as in \emph{Definition \ref{def:influenceset}}. Then for $x\in B^{*c}$ and $2^j<r^{-1}$ one has for any $N>0$
\nm{eq:operatorestimate1}{\abs{T_j^\nu a(x)}\lesssim {  \frac{ 2^{j\brkt{m+\frac{n+1}2}}2^{j M_1}r^{ M_1}r^{n-\frac np}   }{\brkt{1+2^{2j}\abs{\nabla_{\xi_1}\varphi(x,\xi_j^\nu)-\bar y_1}^2}^{N}\brkt{1+2^{j}\abs{\nabla_{\xi'}\varphi(x,\xi_j^\nu)-\bar y'}^2}^{N}}}, } where $\displaystyle M_1$ is any positive integer larger than $\left[ n\brkt{\frac 1p -1}_+ \right ].$ Moreover for $x\in B^{*c}$ and $2^j\geq r^{-1}$ one has
\nm{eq:operatorestimate2}{ \abs{T_j^\nu a(x)}\lesssim {  \frac{2^{j\brkt{m+\frac {n+1}2}}  2^{-jM_2 }r^{-M_2}r^{n-\frac np} 2^{4jN}r^{4N}}{\brkt{1+2^{2j}\abs{\nabla_{\xi_1}\varphi(x,\xi_j^\nu)-\bar y_1}^2}^{N}\brkt{1+2^{j}\abs{\nabla_{\xi'}\varphi(x,\xi_j^\nu)-\bar y'}^2}^{N}}} ,} 
where $M_2>0$ is any positive integer.
\end{lem}

\begin{proof}
We start with \eqref{eq:operatorestimate1}, when $r< 2^{-j}$. Let $p_{j}^{\nu}\brkt{x,y-z}$ be the Taylor polynomial of $K_{j}^{\nu}(x,y)$ of order $M-1$ centred at $y=z$. Then, using the moment conditions of the atom $a$, the result in Lemma \ref{lem:kernelestimate} and Peetre's inequality, we have
\m{
\abs{T_j^\nu a(x)}  &=\abs{\int_{B}(K_{j}^\nu(x,y)-p_{j}^\nu (x,\ y-\overline{y}))a(y)\dd y} \\ &\lesssim  \sum_{ \abs\alpha =M}{\int_{B}\abs{\partial ^\alpha _y K_{j}^\nu(x,\tilde y)}\abs{y-\overline{y}}^{\abs\alpha}\abs{a(y)}\dd y} \\
&\lesssim   {  \int_B\frac{ 2^{j\brkt{m+\frac{n+1}2}}2^{j M}r^{ M}r^{-\frac np}   }{\brkt{ 1+2^{2j}\abs{\nabla_{\xi_1}\varphi(x,\xi_j^\nu)-\tilde y_1}^2}^{N}\brkt{1+2^{j}\abs{\nabla_{\xi'}\varphi(x,\xi_j^\nu)-\tilde y'}^2}^{N}}\dd y}\\
&\lesssim   {  \int_B\frac{ 2^{j(m+\frac{n+1}2)}2^{j M}r^{ M}r^{-\frac np} \brkt{1+ 2^{2j}|\tilde y_1 -\overline{y}_1|^2}^{N}\brkt{1+ 2^{j}|\tilde y' -\overline{y}'|^2}^{N}  }{\brkt{1+2^{2j}\abs{\nabla_{\xi_1}\varphi(x,\xi_j^\nu)-\bar y_1}^2}^{N}\brkt{1+2^{j}\abs{\nabla_{\xi'}\varphi(x,\xi_j^\nu)-\bar y'}^2}^{N}}\dd y} \\& \lesssim {  \int_B\frac{ 2^{j\brkt{m+\frac{n+1}2}}2^{j M}r^{ M}r^{-\frac np} \brkt{1+ 2^{2j}r^2}^{2N}  }{\brkt{1+2^{2j}\abs{\nabla_{\xi_1}\varphi(x,\xi_j^\nu)-\bar y_1}^2}^{N}\brkt{1+2^{j}\abs{\nabla_{\xi'}\varphi(x,\xi_j^\nu)-\bar y'}^2}^{N}}\dd y} \\& \lesssim {  \frac{ 2^{j\brkt{m+\frac{n+1}2}}2^{j M}r^{ M}r^{n-\frac np}   }{\brkt{1+2^{2j}\abs{\nabla_{\xi_1}\varphi(x,\xi_j^\nu)-\bar y_1}^2}^{N}\brkt{1+2^{j}\abs{\nabla_{\xi'}\varphi(x,\xi_j^\nu)-\bar y'}^2}^{N}}},
}
where we have taken $\tilde y$ to be on the line segment between $y$ and $\bar y$, and observe that $\tilde y\in B$.\\

Now we prove \eqref{eq:operatorestimate2}, when $r\geq 2^{-j}.$ Suppose that $k$ is the integer such that \linebreak$2^{-k}\leq r < 2^{-k+1}$. 
By the result in Lemma \ref{lem:kernelestimate}, the observation in Remark \ref{rem:influenceset}(ii) and Peetre's inequality, one has
\m{
|T_{j}a(x)| &\lesssim  {  \int_B\abs{K_j^\nu(x,y)a(y) }\dd y} \\ &\lesssim    {  \int_B\frac{2^{j\brkt{m+\frac {n+1}2}} r^{-\frac np}}{\brkt{1+2^{2j}\abs{\nabla_{\xi_1}\varphi(x,\xi_j^\nu)-y_1}^2}^{N+M}\brkt{1+2^j\abs{\nabla_{\xi'}\varphi(x,\xi_j^\nu)-y'}^2}^{N+M}}\dd y} \\&\lesssim  {  \int_B\frac{2^{j\brkt{m+\frac {n+1}2}} 2^{M(k-j) }r^{-\frac np} 
\brkt{1+ 2^{2j}|y_1 -\overline{y}_1|^2}^{N}\brkt{1+ 2^{j}|y' -\overline{y}'|^2}^{N}
}{\brkt{1+2^{2j}\abs{\nabla_{\xi_1}\varphi(x,\xi_j^\nu)-\bar y_1}^2}^{N}\brkt{1+2^{j}\abs{\nabla_{\xi'}\varphi(x,\xi_j^\nu)-\bar y'}^2}^{N}}\dd y} \\&\lesssim  {  \int_B\frac{2^{j\brkt{m+\frac {n+1}2}} 2^{M(k-j) }r^{-\frac np} 
\brkt{1+ 2^{2j}r^2}^{2N}
}{\brkt{1+2^{2j}\abs{\nabla_{\xi_1}\varphi(x,\xi_j^\nu)-\bar y_1}^2}^{N}\brkt{1+2^{j}\abs{\nabla_{\xi'}\varphi(x,\xi_j^\nu)-\bar y'}^2}^{N}}\dd y} \\&\lesssim  {  \frac{2^{j\brkt{m+\frac {n+1}2}} 2^{-jM} r^{-kM} r^{n-\frac np} 
2^{4jN}r^{4N}
}{\brkt{1+2^{2j}\abs{\nabla_{\xi_1}\varphi(x,\xi_j^\nu)-\bar y_1}^2}^{N}\brkt{1+2^{j}\abs{\nabla_{\xi'}\varphi(x,\xi_j^\nu)-\bar y'}^2}^{N}}}
}
which concludes the lemma.
\end{proof}\hspace*{1cm}

We also prove a lemma that is used for the estimates of $\brkt{T_\sigma^\varphi}^*$ in the proof of Proposition \ref{Thm:high TL}.
\begin{lem}\label{SSS estimates for Kj2}
For $j\geq 1$ and $r<1$ let
\m{K_{j}^{\nu}(x, y):=\int_{\mathbb{R}^{n}}\sigma(y,\xi)\,\psi_j(\xi)\chi_{j}^{\nu}(\xi)\, e^{i(x\cdot\xi -\varphi(y,\xi))} \ddd\xi }
and
\m{
K_j(x,y) = \sum_\nu K_j^\nu (x,y)
}

where  $\sigma \in A^{m_c (1)}(\R^n)$ and $\varphi\in \Phi^2 $ satisfies the $\mathrm{SND}$ condition \eqref{SND}. Then for $y\in B(\bar y,r)$ one has 
\begin{enumerate}
\item[\emph{(i)}]  $\displaystyle \int_{\mathbb{R}^{n}}|K_{j}(x,y)-K_j(x,\bar y)|\dd x\leq A\, 2^{j}\, r,$ \\
\item[\emph{(ii)}]  $\displaystyle \int_{B^{* c}} |K_{j}(x,y)|\dd x\leq A\, 2^{-j}r^{-1},$ for $ 2^{j}\geq r^{-1},$
where $B^\ast$ is as\smallskip\ in \linebreak \emph{Definition \ref{def:influenceset}}.
\end{enumerate}
In all the estimates above, the constant $A$ is independent of $y'$, $\overline{y}$, $j$ and $r$.\\

\end{lem}

\begin{proof}
To prove (i), we use mean value theorem and Lemma \ref{lem:kernelestimate}
\m{
&\int_{\mathbb{R}^{n}}|K_{j}^{\nu}(x,y)-K_j^\nu(x,\bar y)|  \dd x \\ &\leq \int_{\mathbb{R}^{n}}|(y-\bar y)\cdot \nabla_y K_{j}^{\nu}(x,\tilde y)|  \dd x \\ &\lesssim \int_{\mathbb{R}^{n}}\frac{r2^{j}2^{j\brkt{m_c(1)+\frac {n+1}2}}}{\brkt{1+2^{2j}\big|(x_1-\nabla_{\xi_1}\varphi(\tilde y,\xi_j^\nu))\big|^{2}+2^{j}\big|(x'-\nabla_{\xi'}\varphi(\tilde y,\xi_j^\nu))\big|^{2}}^{N}} \dd x \\ &\lesssim
\int_{\mathbb{R}^{n}}\frac{r2^{j}2^{j\brkt{-\frac{n-1}2+\frac {n+1}2}}2^{-j\frac{n+1}2}}{\brkt{1+|x|^{2}}^{N}}\dd x\lesssim  2^{-j\frac{n-1}2}\,2^{j}\,r,
}
where we have taken $\tilde y$ to be on the line segment between $y$ and $\bar y$ and $N$ large enough. Therefore summing in $\nu$ and remembering that there are $O\brkt{2^{j\frac{n-1}2} }$ terms involved, we obtain
$$
\int_{\mathbb{R}^{n}}|K_{j}(x,y)-K_j(x,\bar y)|  \dd x \lesssim \sum_{\nu}2^{-j\frac{n-1}2}\,2^j\,r\lesssim 2^j\,r,
$$
for all $y\in B$; this proves (i) in Lemma \ref{SSS estimates for Kj}.\\

Now to prove (ii) we once again use Lemma \ref{lem:kernelestimate} and the observation in Remark \ref{rem:influenceset} (ii);
\m{
    &\int_{B^{*c}}\abs{K_{j}^{\nu}(x,y)} \dd x\\ \lesssim &\int_{B^{*c}}\frac{2^{j\brkt{m_c(1)+\frac{n+1}2}}}{\brkt{1+2^{2j}\big|(x_1-\nabla_{\xi_1}\varphi(y,\xi_j^\nu))\big|^{2}+2^{j}\big|(x'-\nabla_{\xi'}\varphi(y,\xi_j^\nu))\big|^{2}}^{N+1}}\,\dd x\\ \lesssim &\int_{\R^n}\frac{2^{j\brkt{m_c(1)+\frac{n+1}2}}\, 2^{-j\frac{n+1}2}\, 2^{k-j}}{\brkt{1+|x|^{2}}^{N}}\,\dd x\lesssim \, 2^{-j\frac{n-1}2}\, 2^{-j}\, r^{-1}.
}
 Therefore, summing once again in $\nu$ yields (iii).
\end{proof}\hspace*{1cm}

\section{Ruzhansky-Sugimoto's globalisation technique}\label{RS globalisation}
In \cite{Ruzhansky-Sugimoto}, M. Ruzhansky and M. Sugimoto developed a new technique to transfer local boundedness of Fourier integral operators, which was proven by A. Seeger, C. Sogge and E. Stein \cite{SSS}, to a global result, where the amplitudes of the corresponding operators do not have compact spatial supports. In order to prove global regularity results we follow \cite{Ruzhansky-Sugimoto} and define
$$H(x,y,z)=\inf_{\xi\in\R^n}\abs{z+\nabla_\xi\vartheta(x,y,\xi)}
$$
where for us $\vartheta(x,y,\xi)$ is either $\theta(x,\xi)$ or $-\theta(y,\xi)$, with $\theta\in \Phi^1$ and
\[
\Delta_r:=\{(x,y,z)\in\R^n\times\R^n\times\R^n:H(x,y,z)\geq r\}.
\]
One also defines
\begin{align*}
&\widetilde H(z):=\inf_{x,y\in\R^n}H(x,y,z)=
\inf_{x,y,\xi\in\R^n}\abs{z+\nabla_\xi\vartheta(x,y,\xi)},
\\
&\widetilde\Delta_r:=\set{z\in\R^n:\widetilde H(z)\geq r}.
\end{align*}
and
\begin{align*}
&M_K :=\sum_{|\gamma|\leq K}\sup_{x,y,\xi\in\R^n}
\abs{\langle\xi\rangle^{-(m_{c}-|\gamma|)}\,\partial^\gamma_\xi \sigma(x,y,\xi)},
\\
&N_K :=\sum_{1\leq |\gamma|\leq K}\sup_{x,y,\xi\in\R^n}
\abs{\langle\xi\rangle^{-(1-|\gamma|)}\,\partial^\gamma_\xi \vartheta(x,y,\xi)}.
\end{align*}
We observe that $N_K <\infty$ by the $\Phi^1$ condition on the phase function. Given these definitions one has the following lemma:
\begin{lem}\label{Lem:outside}
Let $r\geq1$ and $K\geq 1$.
Then we have
$\R^n \setminus \widetilde{\Delta}_{2r} \subset \{z;\, |z|<(2+N_K )r \}$. Furthermore for $r>0$, $x\in\widetilde{\Delta}_{2r}$ and $|y|\leq r$
we have
\begin{equation}\label{H bounded by H}
 \widetilde{H}(x)\leq 2H(x,y, x-y)
\end{equation}
and therefore $(x,y,x-y)\in\Delta_r$
\end{lem}
\begin{proof}
For $z\in\R^n\setminus\widetilde{\Delta}_{2r}$,
we have
$\widetilde{H}(z)
<2r$.
Hence, there exist $x_0,y_0,\xi_0\in\R^n$ such that
$$
|z+\nabla_\xi\vartheta(x_0,y_0,\xi_0)|<2r.
$$
Since, $r\geq 1$, this yields that
\[
|z|\leq|z +\nabla_\xi\vartheta(x_0,y_0,\xi_0)|+|\nabla_\xi\vartheta(x_0,y_0,\xi_0)|
\leq 2r+N_K \leq (2+N_K )r.
\]

The claim that $(x,y,x-y)\in\Delta_r$ follows from \eqref{H bounded by H} and the definition of $ \Delta_r$. Therefore it only remains to prove \eqref{H bounded by H}. Now, if $|y|\leq r$ and $x\in \widetilde{\Delta}_{2r}$ then since $  \widetilde{H}(x)\geq 2r$, we have that
\begin{align*}
 \widetilde{H}(x)
&\leq |x+\nabla\vartheta(x,y,\xi)|
\leq|x-y+\nabla\vartheta(x,y,\xi)|+|y|
\\
&\leq|x-y+\nabla\vartheta(x,y,\xi)|+  \frac{\widetilde{H}(x)}2.
\end{align*}

From this, \eqref{H bounded by H} follows at once.
\end{proof}
For proving the global boundedness that we aim to demonstrate, the following result is of particular importance.
\begin{lem}\label{Lem:kernel RuzhSug}
The kernel $$K(x,y,z)=\int_{\R^n}e^{iz\cdot\xi+i\vartheta(x,y,\xi)}\, \sigma (x,y,\xi)\,\ddd\xi$$ is smooth on $\displaystyle \bigcup_{r>0}\Delta_r$.
Moreover, for all $L>n$ and $r\geq 1$ it satisfies
\begin{equation}\label{boundedness of HLK}
\norm{H^{L} K}_{L^\infty(\Delta_r)}\leq C( L,M_L ,N_{L+1}),
\end{equation}
where $C( L, M_L ,N_{L+1})$ is a positive constant depending
only on $L$, $M_L$ and $N_{L+1}$.
For $0< p \leq 1$, $\displaystyle L>\frac np$ and $r\geq 1$, the function $\widetilde{H}(z)$ satisfies the bound

\begin{equation}\label{integrable HL}
\norm{ \widetilde{H}^{-L}}_{L^p\brkt{\widetilde{\Delta}_r}}
\leq C(L,N_{L+1}, p).
\end{equation}
\end{lem}
\begin{proof}
If one introduces the differential operator
$$
D=\frac{(z+\nabla_\xi\vartheta)\cdot\nabla_{\xi}}{i|z+\nabla_\xi\vartheta|^2},
$$
with the transpose $D^*$, then integration by parts $L$ times yields
$$
K(x,y,z)=\int_{\R^n}e^{iz\cdot\xi+i\vartheta(x,y,\xi)}
\left(D^*\right)^{L} \sigma (x,y,\xi)\,\ddd\xi.
$$

Now \eqref{boundedness of HLK} follows from the relation $r\leq H(x,y,z)\leq|z+\nabla_\xi\vartheta(x,y,\xi)|$ which \smallskip is valid for $(x,y,z)\in\Delta_r$ and $\xi\in\R^n$.
Moreover $|z|\leq |z+\nabla_\xi\vartheta(x,y,\xi)|+N_{L+1}$ for any $\xi\not=0$, which yields that $|z|\leq  \widetilde{H}(z)+N_{L+1}$. Hence for $|z|\geq 2N_{L+1}$ one has $\displaystyle |z|\leq \widetilde{H}(z)+\frac{|z|}2$, and therefore
$|z|\leq 2 \widetilde{H}(z)$. Using this we get
\begin{align*}
\norm{ \widetilde{H}^{-L}}_{L^p\brkt{\widetilde{\Delta}_r}}
&\leq
\norm{ \widetilde{H}^{-L}}_{L^p\brkt{\widetilde{\Delta}_r\cap\{|z|\leq 2N_{L+1}\}}}+
\norm{ \widetilde{H}^{-L}}_{L^p\brkt{\widetilde{\Delta}_r\cap\{|z|\geq 2N_{L+1}\}}}
\\
&\leq
r^{-L}\brkt{\int_{|z|\leq 2N_{L+1}}  \dd z}^{\frac 1p} +
2^{L}\brkt{\int_{|z|\geq 2N_{L+1}}|z|^{-pL}\, \dd z} ^{\frac 1p} \\ & \leq C( L, N_{L+1},p),
\end{align*}
which proves \eqref{integrable HL}.
\end{proof}\hspace*{1cm}

Now in the proof of global boundedness of FIOs that are treated in this paper, we shall use Lemma \ref{phase reduction lem} to bring the operators in question to the form
\m{
   \iint_{\R^n\times\R^n} a(x,\xi)\, e^{i\theta(x,\xi)+i(t(x)-y)\cdot\xi}\, u(y)  \ddd\xi \dd y,
}
or
\m{
   \iint_{\R^n\times\R^n}  a(y,\xi)\, e^{i\theta(y,\xi)+i(t(y)-x)\cdot\xi}\, u(y)  \ddd\xi \dd y,
}
where $\theta \in \Phi^{1},$ and $t(\cdot)$ is an appropriate global diffeomorphism.
Therefore a change of variables and using the invariance of Besov-Lipschitz and Triebel-Lizorkin spaces under suitable diffeomorphisms, will enable us to replace $t(x)$ and $t(y)$ by $x$ and $y$ respectively and utilise  the estimates discussed above, to obtain global boundedness results in various settings.

\section{Boundedness of FIOs on Besov-Lipschitz spaces}\label{boundedness in Besov}

In this section we establish the boundedness of FIO's of all possible scales for Besov-Lipschitz spaces $B^s_{p,q}(\R^n)$ for $-\infty <s<\infty,$ $0<p\leq \infty$ and $0<q\leq \infty.$  The local boundedness results are for amplitudes $a(x,\xi)\in S^m(\R^n)$ and phase functions $\varphi(x,\xi)$ that are positively homogeneous of degree 1 in $\xi$ and satisfy the usual non-degeneracy condition. We will also prove global boundedness results for operators with phase functions in $\Phi^2$ that are SND. For the global results to hold, it is necessary that $\displaystyle p>\frac{n}{n+1}$. At this point, it is appropriate to note that the phase function of the Fourier integral operators are in general singular at the origin, therefore in proving various boundedness results, it behoves one to split the operator in high and low frequency parts. Henceforth we shall divide the regularity results into low and high frequency portions.

\subsection{$L^p$ boundedness of a Littlewood-Paley piece of a Fourier integral operator}\label{section: Lp boundedness for pieces}
\noindent Briefly, the result concerning the $L^p$ boundedness of the Littlewood-Paley pieces of an FIO states that, if the operator in question has an amplitude with frequency support in an annulus of size $\sim 2^j$, $j\in \N$, then that operator is $L^p$ bounded. Moreover, the $L^p$ estimate keeps control of the parameter $j$. This will be crucial when estimating the $L^p$ norm o an FIO within the $B_{p,q}^s$ norm.\\

\begin{prop}\label{prop:JFA}
Let $0< p\leq\infty,$ $m_c(p)$ as in \eqref{eq:criticaldecay}, $a\in  S^m(\R^n)$ for $m\in\R$ and $\varphi(x,\xi) \in\mathcal{C}^{\infty}(\R^n \times\R^n \setminus\{0\}),$ be positively homogeneous of degree one in $\xi$. Assume that $\psi_j$ is as in \emph{Definition \ref{def:LP}} and let $T_j$ be a Littlewood-Paley piece of an $\mathrm{FIO}$ $T_a^\varphi$, which is defined by
\begin{equation}\label{eq:JFA3}
T_jf(x) := \int_{\R^n} e^{i\varphi(x,\xi)}\,a(x,\xi)\,(1-\psi_0(2\xi))\,\psi_j(\xi)\,\widehat f(\xi)\ddd \xi.
\end{equation}
 Then if $\varphi\in \Phi^2$ is $\mathrm{SND}$, one has
\begin{equation} \label{eq:JFA1}
\norm{T_jf}_{L^{p}(\R^n)} \lesssim 2^{j(m-m_c(p))}\norm{\Psi_{j} (D) f}_{L^{p}(\R^n)},
\end{equation}
for $j\in \N$ and $\Psi_j$ as defined in  \emph{Definition \ref{def:LP}}. Furthermore, if one assumes that the amplitude $a(x,\xi)$ is compactly supported in $x$, then one has the same result, if the phase function $\varphi$ is assumed to be non-degenerate on the support of $a(x,\xi)$.
\end{prop}
\begin{rem}
A careful examination of the proof of the kernel estimates also reveals that \emph{Proposition \ref{prop:JFA}} is valid in the range $0<p\leq 2$ even if $a\in L^\infty S^m(\R^n)$.
\end{rem}
\begin{rem}
The factor $(1-\psi_0(2\xi))$ is inserted in \eqref{eq:JFA3} to cut off the singularity at $\xi=0$ for the case $j=0$. The singularity has to be taken care of separately and this is done in \emph{Propositions \ref{prop:low1 BL}, \ref{prop:low2 BL}} below.
\end{rem}

\begin{rem}
Note that in the Banach cases, i.e. $p\in [1,\infty]$, \eqref{eq:JFA1} is equivalent to the $L^p$ boundedness of operators $T_j$. However in the quasi-Banach cases, i.e $p<1,$ then one can not get rid of the frequency localisation $\Psi_j (D)$, since any $L^p$ bounded translation invariant operator \emph{(}for $0<p<1$\emph{)} is an infinite linear combination \emph{(}with coefficients in $\ell^p$\emph{)} of Dirac measures, see \cite{Oberlin}.
\end{rem}

\begin{proof}[Proof of \emph{Proposition \ref{prop:JFA}}]
Since the proof is rather lengthy and contains several cases, we split it into four steps as follows;\\
\begin{enumerate}
\item In Step 1 we use the kernel estimate from Lemma \ref{lem:kernelestimate} and prove the proposition for the case $0<p\leq 1$.
\item In Step 2 we once again use Lemma \ref{lem:kernelestimate} to obtain the result for $p=\infty$.
\item In Step 3 we deal with the case of $p=2$.
\item In Step 4 we show the result for the cases $1<p<2$ and $2<p<\infty$, and finally interpolation yields the boundedness for the range $1<p\leq \infty$.
\end{enumerate}\hspace*{1cm}\\
Note that in the proofs of (ii), (iii) and (iv), it will be enough to show an estimate of the form $$\norm{T_j f}_{L^{p}(\R^n)} \lesssim 2^{j\left (m-m_c(p)\right )}\norm{f}_{L^{p}(\R^n)},$$
where we could without any cost, insert a frequency localisation on the right hand side of the estimate above.\\

\newcounter{counter}
\stepcounter{counter}
\noindent \textbf{Step \thecounter\ -- Proof of the case $\mathbf{0<p\leq 1\, }$}\\
We will use the partition of unity \eqref{eq:quadrseconddyad} and decompose the operator $T_j$ as $\displaystyle T_j = \sum_{\nu}T_{j}^{\nu},$
where
\begin{equation*}
\begin{split}
T_{j}^{\nu}f(x):\! &=\iint_{\mathbb{R}^{n}\times \R^n }a(x,\xi)\,(1-\psi_0(2\xi))\,\psi_j(\xi)\,\tilde{\chi}_{j}^{\nu}( \xi)^2\, e^{i\varphi(x,\xi)-iy\cdot\xi}\, f(y)\ddd\xi \dd y
\\ &=\int_{\mathbb{R}^{n}}K_{j}^{\nu}(x,y)\, \mathcal X_j^\nu (D)\,\Psi_j(D)\,f(y)\dd y,\\
 K_{j}^{\nu}(x, y)  :\! &=\int_{\mathbb{R}^{n}}a(x,\xi)\,(1-\psi_0(2\xi))\,\psi_j(\xi)\,\tilde{\chi}_{j}^{\nu}(\xi)\, e^{i\varphi(x,\xi)- iy\cdot\xi}\ddd\xi,
\end{split}
\end{equation*}
where $\mathcal X_j^\nu(D) := \tilde{\chi}_j^\nu(D)\, \Psi_j(D)$ with $\tilde{\chi}_j^\nu$ as in \eqref{chitilde} and $\Psi_j(D)$ as in Definition \ref{def:LP}.
Using the properties \eqref{eq:quadrseconddyadcond} and \eqref{eq:quadrseconddyadcond2} which are also valid for $\tilde{\chi}_j^\nu$, one can verify that the kernel $K_j^\nu$ satisfies \eqref{eq:claimkernel} for $\varPhi(x,y,\xi)= \varphi(x,\xi) -y\cdot\xi$. Now set $f_j^\nu :=  \mathcal X_j^\nu (D)  \Psi_j(D)f$ and
\m{\mathsf f_{j}^\nu (z) :\! & =\sup_{y\in \R^n} \left  (1+\big |2^{j}(z_1-y_1)\big |^{2}\right )^{-M}\left  (1+\big |2^{\frac j2}(z'-y')\big |^{2}\right )^{-M} \abs{ f_j^\nu (y)}\\ &\lesssim \sup_{y\in \R^n} \left (1+\big |2^{j}(z_1-y_1)\big |^{2M}\right )^{-1} \left (1+\big |2^{\frac j2}(z'-y')\big |^{2M}\right)^{-1} \abs{f_j^\nu (y)}.
}
Since $\Psi_j \equiv 1$ on the support of $\psi_j$ we have (using \eqref{eq:claimkernel})
\begin{equation*}
	\begin{split}
		\abs{T_{j}^\nu f(x)} &\leq \int_{\R^n} \abs{K_j^\nu (x,y)f_j^\nu(y)}\dd y \\
        &\lesssim  2^{jm}2^{j\frac{n+1}{2}}\mathsf{f}_{j}^{\nu} (\nabla_\xi \varphi(x,\xi_j^\nu)) \int_{\R^n} \left  (1+\big |2^{j}(\nabla_{\xi_1}\varphi(x,\xi_{j}^{\nu})-y_1)\big |^{2}\right )^{M-N}\\ &\times \left (1+\big |2^{\frac j2}(\nabla_{\xi'}\varphi(x,\xi_{j}^{\nu})-y')\big |^{2}\right )^{M-N} \dd y \\
         &\lesssim  2^{jm}2^{j\frac{n+1}{2}}2^{-j\frac{n+1}{2}} \mathsf f_{j}^\nu (\nabla_\xi \varphi(x,\xi_j^\nu))  = 2^{jm} \mathsf f_{j}^\nu (\nabla_\xi \varphi(x,\xi_j^\nu)),
	\end{split}
\end{equation*}
where $\displaystyle M>\frac 1{2p}$, $N-M>n.$\\

Now in Lemma \ref{lem:Triebel-Schmeisser}, take $k=\nu$, $n'=1$ and $\displaystyle r_1=r_2 = \frac 1{2M} <p$ and note that \\$\supp \widehat {f_j^\nu} \subset \left \{ (\xi_1,\xi')\in \R \times \R^{n-1}:\ |\xi_1|\leq 2^j, |\xi'|\leq 2^{\frac j2} \right \}$. Moreover take $c_{j,\nu}=2^j$ and $d_{j,\nu}=2^{\frac j2}$. Then the conditions of Lemma \ref{lem:Triebel-Schmeisser} all hold for $f_j^\nu$ and therefore we have
\begin{equation*}
\left |T_j^\nu f  (x  ) \right |\leq 2^{jm} \mathsf f_{j}^\nu (\nabla_\xi \varphi(x,\xi_j^\nu))\lesssim 2^{jm}\left ( M_{2}\left (M_1|f_j^\nu|^{r_1} \right )^{\frac{r_2}{r_1}}\right ) ^{\frac{1}{r_2}} \left (\nabla_\xi \varphi(x,\xi_j^\nu ) \right ).
\end{equation*}
Taking the $L^p$ norm of the expression above, and using the SND condition on the phase function and changes of variables, the boundedness of the maximal operators $M_1$ and $M_{2}$ yields that
\begin{equation}\label{eq:highfreq2}
\begin{split}
 \norm{T_j^\nu f}_{L^p(\R^n)}&\lesssim  2^{jm}\norm{M_{2}(M_1|f_j^\nu|^{r_1})^{\frac{r_2}{r_1}}}^{\frac 1{r_2}}_{L^{\frac{p}{r_2}}(\R^n)}\\ &\lesssim
2^{jm}\norm{(M_1|f_j^\nu|^{r_1})}_{L^{\frac{p}{r_1}(\R^n)}}^{\frac 1{r_1}}\lesssim 2^{jm}\norm{f_j^\nu}_{L^p(\R^n)}.
\end{split}
\end{equation}
Here we observe that $f_j^\nu$ can be written of the form $f_j^\nu(x) = \left ((\mathcal X_j^\nu)^\vee *  \Psi_j(D)f \right )(x)$. Therefore, Lemma \ref{lem:hassesfaltning} yields

\begin{equation}\label{eq:highfreq3}
\begin{split}
\norm{f_j^\nu}_{L^p(\R^n)} \lesssim  2^{jn\left (\frac 1p-1 \right )}\norm{(\mathcal X_j^\nu) ^\vee}_{L^p(\R^n)}\norm{ \Psi_j(D)f}_{L^p(\R^n)}.
\end{split}
\end{equation}
Now we would like to estimate $\norm{(\mathcal X_j^\nu) ^\vee}_{L^p(\R^n)}$. Indeed, using \eqref{eq:quadrseconddyadcond} and \eqref{eq:quadrseconddyadcond2},
integration by parts $N$ times yields
\begin{equation*}
	\left (1+\big |2^{j}z_1\big |^2 \right )^{N}\left (1+\big |2^{\frac j2}z'\big |^2 \right )^{N} \left|\brkt{\mathcal X_{j}^{\nu}}^\vee(z)\right|\lesssim \int_{\supp {\mathcal X_j^\nu}} \abs{L^N \mathcal X_j^\nu(\xi)}\ddd\xi \lesssim 2^{j\frac{n+1}{2}},
\end{equation*}
where we have used that $\left |\supp{\mathcal X_j^\nu} \right | = O\brkt{2^{j\frac{n+1}{2}}}$. Hence, it follows that
\begin{equation} \label{eq:highfreq4}
\begin{split}
	\norm{\brkt{\mathcal X_{j}^{\nu}}^\vee}_{L^p(\R^n)} &\lesssim 2^{j\frac{n+1}{2}} \left ( \int_{\R^n} \frac{1}{\left (\left (1+\big |2^{j}z_1\big |^2\right )^{N}\left (1+\big |2^{\frac j2}z'\big |^2\right )^{N}\right )^p} \dd z \right )^{\frac 1p} \\ &=2^{j\frac{n+1}{2}} \left ( \int_{\R^n} \frac{2^{-j}2^{-j\frac{n-1}{2}}}{\left (\left (1+\big |z_1\big |^2\right )^{N}\left (1+\big |z'\big |^2\right )^{N}\right )^p} \dd z \right )^{\frac 1p} \\ & \lesssim  2^{j\left (\frac{n+1}{2}-\frac{n+1}{2p}\right )},
\end{split}
\end{equation}
for $N>n$. Inserting (\ref{eq:highfreq4}) in (\ref{eq:highfreq3}) and then (\ref{eq:highfreq3}) into (\ref{eq:highfreq2}) one has
\m{
\norm{T_j^\nu f}_{L^p(\R^n)} &\lesssim  2^{jm}2^{jn\left (\frac 1p-1 \right )}2^{j\left (\frac{n+1}{2}-\frac{n+1}{2p}\right )}\norm{\Psi_j(D)f}_{L^p(\R^n)}\\ &=2^{j\brkt{m+\frac{n-1}{2p}-\frac{n-1}{2}}} \norm {\Psi_j(D)f}_{L^p(\R^n)}.
}
Summing in $\nu$ (note that there are $O\brkt{2^{j\frac{n-1}{2}}}$ terms involved)
\begin{equation*}
\begin{split}
\norm{T_j f}_{L^p(\R^n)} & \leq  \left (\sum_\nu \norm{T_j^\nu f}_{L^p(\R^n)}^p \right )^{\frac 1p}  \\ &\lesssim \left (\sum_\nu 2^{j\brkt{mp+\frac{n-1}{2}-p\frac{n-1}{2}}}  \norm{\Psi_j (D) f}_{L^p(\R^n)}^p\right )^{\frac 1p}\\ & \lesssim  2^{j\left (m+\frac{n-1}{p}-\frac{n-1}{2}\right )}  \norm{\Psi_j (D)f }_{L^p(\R^n)}\\ & = 2^{j(m-m_c(p))}  \norm{\Psi_j (D)f }_{L^p(\R^n)},
\end{split}
\end{equation*}
and hence the proposition is proven for $0<p\leq1$.\\

\stepcounter{counter}

\noindent \textbf{Step \thecounter\ -- Proof of the case $\mathbf{p= }\, \boldsymbol{\infty} \,$}  \\
Once again we decompose $\R^n$ into cones as in Definition \ref{def:LP2}. This time the partition of unity $\chi_j^\nu$ defined in \eqref{the first partition of unity}. We then decompose $T_j$ as $ \displaystyle T_j = \sum_{\nu}T^{\nu}_{j} $, where
\m{
T_j^\nu f(x) := \int_{{\R^n}} e^{i\varphi(x,\xi)}\,(1-\psi_0(2\xi))\,\psi_j(\xi)\,\chi_j^\nu(\xi)\,a(x,\xi)\,\widehat f(\xi) \ddd\xi = \int_{\R^n} K_j^\nu (x,y)f(y)\dd y,
}
for
\m{
K_j^\nu (x,y) := \int_{\R^n} e^{i\varphi(x,\xi)-iy\cdot\xi}\,(1-\psi_0(2\xi))\,\psi_j(\xi)\,\chi_j^\nu(\xi)\,a(x,\xi)\ddd\xi.
}
This yields
\nm{eq:finfty1}{
\abs{T_j^\nu f(x)} \leq \norm{K_j^\nu(x,\cdot)}_{L^1(\R^n)}\norm{f}_{L^\infty(\R^n)}
}
Once again we have that $K_j^\nu$ satisfies \eqref{eq:claimkernel}, and by a change of variables 
\m{\norm{K_j^\nu(x,\cdot)}_{L^1(\R^n)} &\lesssim 2^{j\brkt{m+\frac{n+1}{2}}}\int_{\R^n}\left (1+\big|2^{j}(\nabla_{\xi_1}\varphi(x,\xi_{j}^{\nu})-y_1)\big|^{2}\right )^{-N}\\ &\times \left (1+\big |2^{\frac j2}(\nabla_{\xi'}\varphi(x,\xi_{j}^{\nu})-y')\big|^{2}\right )^{-N} \dd y \lesssim  2^{jm}.}
Hence the left hand side of \eqref{eq:finfty1} is bounded by $2^{jm}\norm f_{L^\infty(\R^n)}$ uniformly in $x$. Using the fact that there are roughly $O\brkt{2^{j\frac{n-1}2}}$ terms in the sum in $\nu$, \m{\norm{T_jf}_{L^\infty(\R^n)} \lesssim \sum_{\nu}\norm{T_j^\nu f}_{L^\infty(\R^n)} \lesssim 2^{j\brkt{m+\frac{n-1}2}}\norm f_{L^\infty(\R^n)} = 2^{j\brkt{m-m_c(\infty )}}\norm f_{L^\infty(\R^n)}}
and hence the proposition, when $p=\infty$, is proven.\\

\stepcounter{counter}
\noindent \textbf{Step \thecounter\ -- Proof of the case $\mathbf{p=2\, }$}\\
We proceed by studying the boundedness of $S_j:=T_j\circ T_{j}^{\ast}$.  A simple calculation shows that $\displaystyle S_j f(x)=\int_{\R^n} K_j(x,y) f(y)\dd y$ with
 \begin{equation*}
 K_j(x,y):=\int_{\R^n} e^{i(\varphi(x,\xi)-\varphi(y,\xi))}\, (1-\psi_0(2\xi))^2\, \psi_j(\xi)^2\,a(x,\xi)\, \overline{a(y,\xi)}\ddd\xi.
 \end{equation*}
Now since $\varphi$ is homogeneous of degree $1$ in the $\xi$ variable, $K_j (x,y)$ can be written as
\begin{equation*}
K_{j}(x,y)=2^{jn}\int_{\R^n} b_{j}(x,y,2^{j}\xi)
e^{i2^j \Phi(x,y,\xi)}\ddd\xi.
\end{equation*}
with
\m{\Phi(x,y,\xi)&:= \varphi (x,\xi) -\varphi (y,\xi), \\ b_j\left (x,y,\xi \right )&:={(1-\psi_0(2\xi))^2}\,\psi_j(\xi)^2\,a(x,\xi)\, \overline{a(y,\xi)}.
} Observe that the $\xi$-support of $b_{j}(x,y,2^{j}\xi)$ lies in the compact set $\displaystyle \mathcal{K}:=\left \{\frac{1}{2}\leq \abs{\xi}\leq 2\right \}$.  From the $\Phi^2$ and SND conditions \eqref{SND} it also follows that
\begin{equation}\label{Phi cond 1}
\vert \nabla_{\xi}\Phi (x,y, \xi)\vert \approx \vert x-y\vert, \quad \text{for any $x,y\in \R^n$ and $\xi\in \mathcal{K}$}.
\end{equation}Assume that $M>n$ is an integer, fix $x\neq y$ and set $\phi(\xi):=\Phi (x,y, \xi)$, $\Psi:=\abs{\nabla_\xi \phi}^2$. By the mean value theorem, \eqref{C_alpha} and \eqref{Phi cond 1}, for any multi-index $\alpha$ with $\abs{\alpha}\geq 1$ and any $\xi\in \mathcal{K}$,
\[
    \abs{\d^\alpha_{\xi} \phi(\xi)}\lesssim \vert \nabla_{\xi}\Phi (x,y, \xi)\vert= \Psi^{\frac 12}.
\]
On the other hand, since $\displaystyle \d^\alpha_{\xi} \Psi=\sum_{j=1}^n \sum_{\beta\leq\alpha}  \binom{\alpha}{\beta} \d^\beta_{\xi}\d_{\xi_j} \phi \d^{\alpha-\beta}_{\xi}\d_{\xi_j}  \phi$, it follows that, for any $\abs{\alpha}\geq 0$, $\abs{\d^\alpha_{\xi} \Psi}\lesssim \Psi$. We estimate the kernel $K_j$ in two different ways. For the first estimate, \eqref{Phi cond 1} and Lemma \ref{technic} with $F=b_{j} (x,y,2^{j}\xi  ),$ yield
\begin{equation}\label{eq:JFAest1}
\begin{split}
&\vert K_{j}(x,y)\vert \\  \leq &\,2^{j n} 2^{-j M}\ C_{M,\mathcal{K}} \sum _{\vert \alpha \vert\leq M} 2^{j\abs{\alpha}}\int_{\R^n} {\left \vert \partial^{\alpha}_{\xi} b_{j}(x,y,2^j \xi) \right \vert \big \vert \nabla_{\xi}\Phi(x,y,\xi) \big \vert^{-M}} \ddd\xi \\
\lesssim & \,2^{-j M} \abs{x-y}^{-M} \sum _{\vert \alpha \vert\leq M}  2^{j\abs{\alpha}} \int_{\R^n} \abs{\partial^{\alpha}_{\xi} b_{j}(x,y,\xi)}\ddd\xi \\ \lesssim &\, 2^{j(2m+n)}\brkt{2^{j } \abs{x-y}}^{-M}.
\end{split}
\end{equation}
where the fact that the $\xi$ support of $b_j$ lies in a ball of radius $\sim 2^j$ and that
\begin{equation}\label{eq:mk}
	\abs{\partial^{\alpha}_{\xi} b_{j}(x,y,\xi)} \lesssim 2^{j(2m-|\alpha|)},
\end{equation} have been used. Using \eqref{eq:mk} we also obtain
\begin{equation}\label{eq:JFAest2}
\vert K_{j}(x,y)\vert \leq {\int_{\R^n} \abs{b_j(x,y,\xi)}\ddd\xi} \lesssim 2^{j(2m+n)},
\end{equation}
and when combining estimates \eqref{eq:JFAest1} and \eqref{eq:JFAest2} one has
\begin{equation}\label{eq:Kk}
	\vert K_{j}(x,y)\vert\lesssim 2^{j(2m+n)} \brkt{1+2^j\abs{x-y}}^{-M}.
\end{equation}
Thus, using \eqref{eq:Kk} and Minkowski's inequality we have \m{\norm{S_jf}_{L^2(\R^n)} \lesssim 2^{j(2m+n)}\norm{\, \int_{\R^n} \brkt{1+2^j|y|}^{-M}f(\,\cdot-y)\dd y\, }_{L^2(\R^n)} \lesssim 2^{2jm}\norm{f}_{L^2(\R^n)}. }
Since $m_c(2)=0$, the Cauchy-Schwarz inequality yields\m{\norm{T_j^{\ast}f}^2_{L^2(\R^n)} = {\left \langle T_j T_j^{\ast}f,f\right\rangle_{L^2(\R^n)}}  \lesssim \norm{S_jf}_{L^2(\R^n)} \norm f_{L^2(\R^n)} = {2^{2j(m-m_c(2))}\norm f^2_{L^2(\R^n)}}.}
Therefore $\displaystyle \norm{T_j}_{L^2(\R^n)\to L^2(\R^n)}=\|T_j^*\|_{L^2(\R^n)\to L^2(\R^n)}\lesssim 2^{j(m-m_c(2))}$ and the proposition is proven for the case $p=2$.\\

\stepcounter{counter}
\noindent \textbf{Step \thecounter\ -- Proof of the case $\mathbf{1<p<2}$ and $\mathbf{2<p<}\, \boldsymbol{ \infty}$}\\
Now that we have the desired result for $p=1$, $p=2$ and $p=\infty$, we can complete the proof of the proposition. Indeed, the Riesz-Thorin interpolation theorem in $1<p<2$ and $2\leq p\leq \infty$ yields that
\begin{equation*}
\norm{T_j f}_{L^{p}(\R^n)} \lesssim 2^{j\left (m-m_c(p)\right )}\norm{f}_{L^{p}(\R^n)},
\end{equation*}
which thereby concludes the proof of Proposition \ref{prop:JFA}, when the amplitude is not compactly supported in $x$.\\

In case $a(x,\xi)$ is compactly supported in $x$, then the homogeneity of the phase, and its non-degeneracy will once again yield all the kernel estimates above, and therefore the proof goes along the exact same lines as in the non-compactly supported case.
\end{proof}

\subsection{Besov-Lipschitz boundedness for the high frequency portion of FIOs }\label{section:the global section}

\noindent In this section we prove the boundedness of FIOs, where the amplitudes are frequency-supported outside the origin. To this end we have the following:

\begin{prop}\label{prop:high}
Let $0< p,q\leq\infty,$ $m_c(p)$ as in \eqref{eq:criticaldecay}, $a\in S^m(\R^n)$ and \linebreak$\varphi(x,\xi) \in\mathcal{C}^{\infty}(\R^n \times\R^n \setminus\{0\}),$ be positively homogeneous of degree one in $\xi$. Then if $\varphi\in \Phi^2$ satisfies the $\mathrm{SND}$ condition \eqref{SND}, then the operator $T_{\sigma}^\varphi$ given by \eqref{definition of FIO} with amplitude $\sigma( x, \xi):= (1-\psi_0(\xi))\,a(x,\xi)$
satisfies $\displaystyle  T_{\sigma}^\varphi: B_{p,q}^{s+m-m_c(p)}(\R^n)\rightarrow B_{p,q}^s(\R^n)$, for any $s\in \R$. Furthermore, if one assumes that the amplitude $a(x,\xi)$ is compactly supported in $x$, then one has the same result, if the phase function $\varphi$ is assumed to be non-degenerate on the support of $a(x,\xi)$.
\end{prop}

\begin{proof}
We divide the proof into three steps. In Step 1 we invoke a composition formula which yields a sum of two terms (a main term and a remainder term) that need to be analysed separately, and conclude that the main term  is $L^p$ bounded (in the sense of Proposition \ref{prop:JFA}). In Step 2 we show $B_{p,q}^s \to L^p$ boundedness for the remainder term and in Step 3 we complete the proof by deducing the $ B_{p,q}^{s+m-m_c(p)}\rightarrow B_{p,q}^s$ boundedness.\\

\noindent \textbf{Step 1 -- a composition formula and boundedness of the main term}\\ In the definition of the Besov-Lipschitz norm, the expression $\psi_j(D)T_{\sigma}^{\varphi}f$ plays a central role. To obtain favourable estimates for $\psi_j(D)T_{\sigma}^{\varphi}f$ we shall use the parameter-dependent composition formula in Theorem \ref{prop:monsteriosity}. According to that formula, for any integer $M\geq 1$ we can write
\begin{equation}\label{eq:highfreq1}
\psi\left (2^{-j}D \right ) T_{\sigma}^\varphi = \sum_{|\alpha|\leq M-1} \frac{2^{-j|\alpha|}}{\alpha!}T^{\varphi}_{\sigma_{\alpha,j}} + 2^{-jM\varepsilon}T^\varphi_{r_j},
\end{equation}
where $\displaystyle 0<\varepsilon<\frac 12$. Observe that we have replaced $t$ by $2^{-j}$ in Theorem \ref{prop:monsteriosity}. Now
\m{
\left |\partial_\xi^\gamma \partial_x^\beta \sigma_{\alpha,j}(x,\xi) \right |&\lesssim 2^{-j|\alpha|(\varepsilon-1)}\la \xi\ra^{m-|\alpha|\left (\frac 12 -\varepsilon \right )-|\gamma|},\\  \supp_\xi \sigma_{\alpha,j}(x,\xi) &= \left \{ \xi\in\R^n: \ C_1 2^j\leq |\xi| \leq C_2 2^j \right \}\text{ and } \\ r_j (x,\xi)&\in S^{m-M\left (\frac 12 -\varepsilon \right )}(\R^n),
}
where we mention in passing that $r_j (x,\xi)$ vanishes in a neighborhood of $\xi=0$. 

\noindent Therefore Proposition \ref{prop:JFA} and change of variables, imply that
\begin{equation} \label{eq:highfreq8}
\begin{split}
\norm{T^{\varphi}_{\sigma_{\alpha,j}} f}_ {L^p(\R^n)} \lesssim 2^{j(m-m_c(p))}   \norm{\Psi_j(D)f}_{L^p(\R^n)}.
\end{split}
\end{equation}

\noindent \textbf{Step 2 -- The remainder term}\\
To deal with the remainder term of \eqref{eq:highfreq1}, we decompose $T^\varphi_{r_j}$ in  into Littlewood-Paley pieces as follows:
\m{T^\varphi_{r_j} f(x) = \sum_{{k=0}}^\infty T^\varphi_{r_{j,k}}f(x),}
where $T^\varphi _{r_{j,k}}$ is an FIO with amplitude $r_{j,k}:= r_j(x,\xi)\,\psi_k(\xi)$ and the $\psi_k$'s are defined in Definition \ref{def:LP}. We use the fact that for $0<p\leq\infty,$
\begin{equation}\label{quasibanch triangel}
\norm{f+g}_{L^p(\R^n)} \leq 2^{C_p}\left (\norm{f}_{L^p(\R^n)}+\norm{g}_{L^p(\R^n)}\right ),
\end{equation}
where $\displaystyle C_p:=\max\left (0,\frac 1p -1\right )$.
Now Fatou's lemma and iteration of \eqref{quasibanch triangel} yield that
\begin{equation*}
\begin{split}
\norm{T^\varphi_{r_{j}} f}_{L^p(\R^n)}&=\norm{ \sum_{k=0}^\infty T^\varphi_{r_{j,k}}f}_{L^p(\R^n)} \leq \liminf_{N\to\infty}\norm{\sum_{k=0}^N T^\varphi_{r_{j,k}}f}_{L^p(\R^n)} \\& \lesssim \liminf_{N\to\infty}\sum_{k=0}^N2^{kC_p}\norm{ T^\varphi_{r_{j,k}}f}_{L^p(\R^n)} \lesssim\sum_{k=0}^\infty 2^{kC_p}\norm{T^\varphi_{r_{j,k}}f}_{L^p(\R^n)},
\end{split}
\end{equation*}
where the hidden constant in the last estimate only depend on $p$. Therefore, applying Proposition \ref{prop:JFA} with $\displaystyle m-M\left (\frac 12 - \varepsilon \right )$ instead of $m$ (recall that $r$ vanishes in a neighborhood of $\xi=0$), we obtain
\begin{equation}\label{eq:highfreq6}
\begin{split}
\norm{T^\varphi_{r_{j}}f}_{L^p(\R^n)} &\lesssim \sum_{k=0}^\infty 2^{kC_p}\norm{T^\varphi_{r_{j,k}}f}_{L^p(\R^n)} \\ &\lesssim \sum_{k=0}^\infty 2^{k\left (C_p+m-m_c(p)-M\left (\frac 12 - \varepsilon \right ) \right )}\norm{\Psi_k(D)f}_{L^p(\R^n)}.
\end{split}
\end{equation}
Note that the estimate \eqref{eq:highfreq6} is uniform in $j$. Now take
\begin{equation}\label{eq:highfreq7}
M>\max \left (\frac{C_p-s}{\frac 12-\varepsilon },\frac s\varepsilon \right ).
\end{equation}
Then we claim that
\begin{equation}\label{eq:highfreq9}
T^\varphi_{r_{j}}:B_{p,q}^s(\R^n)\to L^p(\R^n).
\end{equation}
To see this, we shall analyse the cases $0<q<1$ and $1\leq q \leq \infty$ separately. Starting with the former, we have
\begin{equation*}
\begin{split}
\norm{T^\varphi_{r_{j}}f}_{L^p(\R^n)} \lesssim &\sum_{k=0}^\infty 2^{k\left (C_p+m-m_c(p)-M\left (\frac 12 - \varepsilon \right ) \right )}\norm{\Psi_k(D)f}_{L^p(\R^n)}  \\  \lesssim & \sum_{k=0}^\infty 2^{k(s+m-m_c(p))}\norm{\Psi_k(D)f}_{L^p(\R^n)} \\ \leq &\left (\sum_{k=0}^\infty 2^{kq(s+m-m_c(p))}\norm{\Psi_k(D)f}_{L^p(\R^n)}^q\right )^{\frac 1q}= \norm{f}_{B_{p,q}^{s+m-m_c(p)}(\R^n)},
\end{split}
\end{equation*}
where we used (\ref{eq:highfreq6}) for the first inequality and (\ref{eq:highfreq7}) for the second. For $1\leq q \leq \infty$ we have in a similar way
\begin{equation*}
\begin{split}
\norm{T^\varphi_{r_{j}} f}_{L^p(\R^n)} &\lesssim \sum_{k=0}^\infty 2^{k\left (C_p+m-m_c(p)-M\left (\frac 12 - \varepsilon \right ) \right )}\norm{\Psi_k(D)f}_{L^p(\R^n)} \\ &= \sum_{k=0}^\infty 2^{k\left (-s +C_p-M\left (\frac 12 - \varepsilon \right ) \right )}\brkt{2^{k(s+m-m_c(p))}\norm{\Psi_k(D)f}_{L^p(\R^n)}}  \\ &\lesssim \left ( \sum_{k=0}^\infty 2^{kq'\left (-s +C_p-M\left (\frac 12 - \varepsilon \right ) \right )}\right )^{\frac 1{q'}} \left ( \sum_{k=0}^\infty 2^{kq(s+m-m_c(p))}\norm{\Psi_k(D)f}^{q}_{L^p(\R^n)}\right )^{\frac 1q} \\ & \lesssim \norm{f}_{B_{p,q}^{s+m-m_c(p)}(\R^n)}
\end{split}
\end{equation*}
and the claim (\ref{eq:highfreq9}) is proven. Note that the calculation above also holds for $q=\infty$ with the usual interpretation of H\"older's inequality. \\\\
\textbf{Step 3 -- The $\mathbf{B_{p,q}^{s+m-m_c(p)}\to B_{p,q}^s}$ boundedness}\\
The results in (\ref{eq:highfreq8}) and (\ref{eq:highfreq9}) yield that
\begin{equation*}\label{eq:highfreq10}
\begin{split}
&\norm{T_{\sigma}^\varphi f}_{B_{p,q}^{s}(\R^n)} \\=& \left (\sum_{j=0}^\infty \left (2^{js}\norm{\psi\brkt{2^{-j}D}T_{\sigma}^\varphi f}_{L^p(\R^n)}\right )^q\right )^{\frac 1q}\\ \lesssim&  \left (\sum_{j=0}^\infty \left (\sum_{|\alpha|\leq M-1} 2^{js}\norm{T^\varphi_{\sigma_{\alpha,j}} f}_{L^p(\R^n)} + 2^{-j(M\varepsilon-s)}\norm{T^\varphi_{r_j} f}_{L^p(\R^n)}\right )^q \right )  ^{\frac 1q}  \\ \lesssim&  \brkt{ \sum_{j=0}^\infty \left ( 2^{j(s+m-m_c(p))}\norm{ \Psi_j(D)f}_{L^p(\R^n)} + 2^{-j(M\varepsilon-s)} \norm{ f}_{B_{p,q}^{s+m-m_c(p)}(\R^n)}\right  )^q}^{\frac 1q} \\ \lesssim& \left (\sum_{j=0}^\infty  2^{jq(s+m-m_c(p))}\norm{ \Psi_j(D)f}_{L^p(\R^n)}^q + \sum_{j=0}^{\infty}2^{-jq(M\varepsilon-s)}\norm{ f}_{B_{p,q}^{s+m-m_c(p)}(\R^n)}^q \right )^{\frac 1q}\\ \lesssim& \left (\norm{ f}_{B_{p,q}^{s+m-m_c(p)}(\R^n)}^q\right )^\frac 1q = \norm{ f}_{B_{p,q}^{s+m-m_c(p)}(\R^n)},
\end{split}
\end{equation*}

and the proof is complete.
\end{proof}\hspace*{1cm}

\subsection{Besov-Lipschitz boundedness of the low frequency portion of FIOs}\label{section:the local section Besov}
\noindent In this section we prove the boundedness of FIOs, where the amplitudes are frequency-supported in a neighbourhood of the origin. In this case, we will need to distinguish between two cases. First we assume that  the amplitude of our FIO is compactly supported in the $x$-variable. This extra assumption enables us to prove the boundedness for the whole range $0<p\leq \infty$. In the second case, we remove the assumption of compact support in the spatial variable on the amplitude. In this case it turns out that we have to confine ourselves to the range $\displaystyle \frac{n}{n+1}<p\leq\infty$. We start with the local result.

In what follows we let $T_{a_0}^\varphi$ denote an FIO with amplitude $a_0(x,\xi):=a(x,\xi)\psi_0(\xi),$ where $\psi_0$ is as in Definition \ref{def:LP}.

\begin{prop}[Local boundedness]\label{prop:low1 BL}
Let $a(x,\xi)\in S^m(\R^n)$ be compactly supported in the $x$ variable and let $\varphi(x,\xi) \in\mathcal{C}^{\infty}(\R^n \times\R^n \setminus\{0\}),$ be positively homogeneous of degree one in $\xi$, and non-degenerate on the support of $a(x,\xi)$. Then $T_{a_0}^\varphi :B_{p,q_1}^{s_1}(\R^n)\to B_{p,q_2}^{s_2}(\R^n),$  for any $s_1, s_2\in (-\infty, \infty)$, and $p, q_1, q_2 \in (0,\infty]$.
\end{prop}

\begin{proof}
First we use Lemma \ref{phase reduction lem} to reduce the operator to finite sums of operators of the form
$$\int_{\R^n} a_0 (x,\xi)\, e^{i\theta(x,\xi)+i\nabla_{\xi}\varphi(x,\zeta)\cdot\xi}\, \widehat{u}(\xi) \, \ddd\xi$$
where $\zeta$ is a point on the unit sphere $\mathbb{S}^{n-1}$, $\theta(x,\xi)\in \Phi^{1},$ and $a_{0}(x,\xi)\in S^{m}(\R^n)$ is localised in the $\xi$ variable around the point $\zeta$. Then observe that if $t(x)= \nabla_\xi \varphi(x,\zeta)$, then due to the SND condition on the phase, $t(x)$ is a global diffeomorphism and the Jacobian matrix of $t(x)$, $Dt(x)=\Big(\partial^2_{x_{j}\xi_{k}} \varphi (x,\zeta)\Big),$ has bounded entries (by the $\varphi\in \Phi^2$ assumption) and hence $\abs{\det Dt(x)}\lesssim 1$.\\

This enables us to use the invariance of Besov-Lipschitz spaces under diffeomorphisms (Theorem \ref{invariance thm}) to reduce the proof of the proposition, to the case of operators $T_{a_0}^{\varphi}$ with $a_0\in S^m$ and $\varphi(x,\xi)=x\cdot\xi+\theta(x,\xi)$ with $\theta \in \Phi^1$.\\

Without loss of generality we can assume that $f=\chi(D) f$ where $\chi$ is a smooth cut-off function that is equal to one on the support of $\psi_0$. Define the self-adjoint operators
\begin{equation*}
L_\xi := 1-\Delta_\xi \text{ and } L_y := 1-\Delta_y,
\end{equation*}
and note that
\begin{equation*}
\jap{\xi}^{-2} L_y \,e^{i( x-y)\cdot\xi } =  \jap{x-y}^{-2} L_\xi\, e^{i( x-y)\cdot\xi } =e^{i( x-y)\cdot\xi }
\end{equation*}
Take integers $\displaystyle N_1>\frac {s_2+n}2$ and $\displaystyle N_2>\frac{n}{2p}$. Integrating by parts, we have
\begin{equation}\label{eq:lowfreq3}
\begin{split}
&\psi_j(D)T_{a_0}^{\varphi}f(x) = \iint_{\R^{n}\times \R^n } e^{i( x-y)\cdot\xi } \psi_j(\xi)T_{a_0}^\varphi f(y)\dd y \ddd \xi  \\ = & \iint_{\R^{n}\times \R^n}  \jap{\xi}^{-2N_1} L_y^{N_1} \left (\jap{x-y}^{-2N_2} L_\xi^{N_2} e^{i( x-y)\cdot\xi }\right ) \psi_j(\xi)T_{a_0}^\varphi f(y)\dd y \ddd \xi  \\ =&
\iint_{\R^{n}\times \R^n }     e^{i( x-y)\cdot \xi }   L_\xi^{N_2} \left (\jap{\xi}^{-2N_1}{\psi_j(\xi)}\right ) \jap{x-y}^{-2N_2} L_y^{N_1} T_{a_0}^\varphi  f(y)\dd y \ddd \xi.
\end{split}
\end{equation}
Since $\psi_j$ is supported on an annulus of size $2^j$ one has
\begin{equation}\label{eq:lowfreq4}
\begin{split}
\int_{\R^n} \left | L_\xi^{N_2} \jap{\xi}^{-2N_1}{\psi_j(\xi)} \right | \ddd \xi &\lesssim \sum_{|\alpha|\leq 2N_2} \int_{|\xi| \sim 2^j}  \left |\partial_\xi^\alpha \Big (\la \xi \ra ^{-2N_1}\psi_j(\xi) \Big ) \right |\ddd \xi
\\  &\lesssim 2^{jn} \sum_{|\alpha|\leq 2N_2} 2^{-j(2N_1+|\alpha|)} \lesssim  2^{j(n-2N_1)}
\end{split}
\end{equation}
Also, applying Leibniz's and Fa\`a di Bruno's formulae we have that
\begin{equation} \label{eq:lowfreq5}
\begin{split}
T_{b_{0}}^\varphi  f(y):\! &= L_y^{N_1} T_{a_0}^\varphi f(y) = \int_{\R^n} L_y^{N_1} \Big(a(y,\eta) e^{i\varphi(y,\eta)}\Big)\psi_0(\eta)\, \widehat f(\eta) \ddd   \eta \\ &= \int_{\R^n}
 b_{0} (y,\eta)\, e^{i\varphi(y,\eta)}\, \,\widehat{f}(\eta) \ddd \eta,
\end{split}
\end{equation}
with $$b_{0} (y,\eta) := \sum_{|\alpha|\leq 2N_1 }\sum_{1\leq |\beta|\leq 2N_1 } \sum_{l\leq N_1}  C_{\alpha,\beta, l} \Big (\partial_y^\alpha  a(y,\eta)\Big ) \Big (\partial_y^\beta \varphi(y,\eta) \Big)^l \psi_0(\eta).$$
Observe that the assumption on the phase and the mean-value theorem yield \linebreak$\abs{\partial_y^\beta \varphi(y,\eta)}= \abs{\partial_y^\beta \varphi(y,\eta)-\partial_y^\beta \varphi(y,0)}\lesssim |\eta|,$ for $|\eta|\neq 0$ and $|\beta|\geq 1.$
Thus $T^\varphi_{b_{0}} $ is the same type of FIO as $T^\varphi_a$, and we have
\begin{equation}\label{pointwise estimate for LP piece}
\left |\psi_j(D)T_{a_0}^\varphi f(x) \right | \lesssim
2^{j(n-2N_1)} \Big (\la \cdot \ra ^{-2N_2} * \left |T_{b_{0}} ^\varphi f \right | \Big )(x).
\end{equation}

Now using Lemma \ref{lem:David-Wulf} with $b(x,\xi)=b_{0}(x,\xi)\,e^{i\varphi(x,\xi)-ix\cdot\xi}= b_{0}(x,\xi)\,e^{i\theta(x,\xi)}$, and recalling that $\theta(x,\xi)\in \Phi^{1},$ we can see that the kernel of $T^\varphi_{b_{0}}$ satisfies the estimate
\begin{equation*}
|K(x,y)| \lesssim \la x-y\ra^{-n-\varepsilon}.
\end{equation*}
Therefore, since $f$ is frequency localised, an application of Lemma \ref{grafakos lemma 1} and the fact that $b_{0}$ is compactly supported yield the pointwise estimate
\begin{equation}\label{pointwise Tb}
|T_{b_{0}}^\varphi f (y)|\lesssim \chi_{\mathcal{K}} (y)\, \Big( M \left (|f|^r \right)\Big)^{\frac{1}{r}}(y),
\end{equation}
for $\displaystyle r>\frac n{n+1}$, where $\displaystyle \mathcal{K}=\supp_{y}b_{0}(y,\xi)$.\\

Hence, \eqref{pointwise estimate for LP piece}, \eqref{pointwise Tb} and Peetre's inequality yield
\begin{equation}\label{jaevla namn}
\begin{split}
\left |\psi_j(D)T_{a_0}^\varphi f(x) \right |
&\lesssim 2^{j(n-2N_1)}  \brkt{\la \cdot \ra ^{-2N_2} * \chi_{\mathcal{K}}\,\Big( M \left (|f|^r \right)\Big)^{\frac{1}{r}}}(x)\\&\lesssim 2^{j(n-2N_1)} \la x \ra ^{-2N_2}  \int_{\mathcal{K}} \Big( M \left (|f|^r \right)\Big)^{\frac{1}{r}}(y)\, \dd y.
\end{split}
\end{equation}
Now taking the $L^p$ norm, choosing $N_2$ large enough, using the $L^\infty$ boundedness of the Hardy-Littlewood maximal operator, and finally using Lemma \ref{lem:bernstein}, we obtain for $0<p\leq \infty$
\begin{equation}\label{jaevla namn 2}
\begin{split}
\Vert \psi_j(D)T_{a_0}^\varphi f(x) \Vert_{L^{p}(\R^n)}
 &\lesssim  2^{j(n-2N_1)}\Vert |f|^r \Vert_{L^{\infty}(\R^n)}^{\frac{1}{r}}\lesssim 2^{j(n-2N_1)} \norm{f}_{L^\infty(\R^n)}\\ &\lesssim 2^{j(n-2N_1)} \norm{f}_{L^p(\R^n)}.
\end{split}
\end{equation}
Thus \eqref{jaevla namn 2} yields that
\begin{equation*}
\begin{split}
\norm{T_{a_0}^\varphi f}_{B_{p,q_2}^{s_2}(\R^n)} &= \left (\sum_{j=0}^\infty 2^{js_2q_2}\norm{\psi_j(D)T_{a_0}^\varphi f}_{L^p(\R^n)}^{q_2} \right )^{\frac 1{q_2}} \\ &\lesssim \left (\sum_{j=0}^\infty 2^{jq_2(s_2+n-2N_1)} \norm{f}_{L^{p}(\R^n)}^{q_2} \right )^{\frac 1{q_2}} \\ & =  \norm{f}_{L^{p}(\R^n)}\left (\sum_{j=0}^\infty 2^{jq_2(s_2+n-2N_1)} \right )^{\frac 1{q_2}}\\& \lesssim \norm{ f}_{L^{p}(\R^n)}\lesssim \norm{ f}_{B_{p,q_1}^{s_1}(\R^n)}.
\end{split}
\end{equation*}
\end{proof}
Now we state and prove the global boundedness of FIOs with frequency localised amplitudes on Besov-Lipschitz spaces.
\begin{prop}[Global boundedness]\label{prop:low2 BL}
Let $a(x,\xi)\in S^m(\R^n)$ and $\varphi (x,\xi)\in \Phi^2$ and verifies the $\mathrm{SND}$ condition. Then $T_{a_0}^\varphi :B_{p,q_1}^{s_1}(\R^n)\to B_{p,q_2}^{s_2}(\R^n),$ for any \linebreak$s_1, s_2\in (-\infty, \infty)$, $q_1, q_2 \in (0,\infty]$ and $\displaystyle p\in \left (\frac  n{n+1},\infty \right ]$.
\end{prop}

\begin{proof}

The proof differs only marginally from that of Proposition \ref{prop:low1 BL}. First we once again without loss of generality assume that $f=\chi(D) f$ where $\chi$ is a smooth cut-off function that is equal to one on the support of $\psi_0$. Then considering $\psi_j(D)T_{a_0}^{\varphi}f(x)$ as an oscillatory integral, we can deduce that the integral representation \eqref{eq:lowfreq3} is valid for $\psi_j(D)T_{a_0}^{\varphi}$  even in the current case.  Then once again using Lemma \ref{lem:David-Wulf} with $b(x,\xi)=b_{0}(x,\xi)\,e^{i\theta(x,\xi)}$ ($b_{0}$ is as in Proposition \ref{prop:low1 BL}) and the fact that  $\theta \in \Phi^1$, we can see that the kernel of $T^\varphi_{b_{0}}$ satisfies the estimate
\begin{equation*}
|K(x,y)| \lesssim \la x-y\ra^{-n-\varepsilon}.
\end{equation*}

Moreover, from Lemma \ref{grafakos lemma 1}, it follows that for $\displaystyle r>\frac{n}{n+1}$

\begin{equation*}
\begin{split}
\left |\psi_j(D)T_{a_0}^\varphi f(x) \right |
&\lesssim 2^{j(n-2N_1)} \int_{\R^n} \brkt{\int_{\R^n} \langle x-z\rangle^{-2N_2} \langle z-y\rangle^{-n-\varepsilon} \, \dd z}\, |f(y)|\, \dd y\\&\lesssim 2^{j(n-2N_1)} \int_{\R^n} \langle x-y\rangle^{-n-\varepsilon} \, |f(y)|\, \dd y\lesssim  2^{j(n-2N_1)} \Big( M \left (|f|^r \right)\Big)^{\frac{1}{r}}(x).
\end{split}
\end{equation*}

This yields that for $r<p\leq \infty$ one has

\begin{equation*}
\Vert \psi_j(D)T_{a_0}^\varphi f(x) \Vert_{L^{p}(\R^n)}
 \lesssim 2^{j(n-2N_1)} \norm{f}_{L^p(\R^n)}.
\end{equation*}
 and the proof can be concluded following the same argument as in the proof of Proposition \ref{prop:low1 BL}.
\end{proof}

\subsection{Local and Global boundedness of FIOs on Besov-Lipschitz spaces}\label{section.local and global besov for fios}
\noindent In this section we state and prove the local and global boundedness of Fourier integral operators on Besov-Lipschitz spaces. In view of the results of the previous sections, what remains to do is to basically put all the bits and pieces (i.e. the high and low frequency results for various cases) together. As usual, 	$T_a^\varphi$ denotes an FIO given by \eqref{definition of FIO}.\\

\noindent Our main local and global boundedness results are
\begin{thm}\label{thm:local and global BL}Let $a (x,\xi)\in S^{m}(\R^n)$, $p\in (0,\infty]$ and $\displaystyle m_c(p) := -(n-1)\abs {\frac 1p -\frac 12}.$ Assume also that $\varphi(x,\xi) \in\mathcal{C}^{\infty}(\R^n \times\R^n \setminus\{0\}),$ is positively homogeneous of degree one in $\xi$. Then under these assumptions, the following results hold true$:$

\begin{enumerate}
\item[\emph{(i)}] If $a(x,\xi)$ has compact support in $x$ and $\varphi(x,\xi)$ is non-degenerate on the support of $a(x,\xi),$
then for any $s\in \R$, $0<p\leq \infty$ and $0<q\leq \infty$
	\begin{equation*}
	T_a^\varphi: B_{p,q}^{s+m-m_c(p)}(\R^n)\rightarrow B_{p,q}^s(\R^n),
	\end{equation*}
	
\item[\emph{(ii)}] If $\varphi (x,\xi)\in \Phi^2$ is $\mathrm{SND}$, then for any $s\in \R$, $\displaystyle \frac{n}{n+1}<p\leq \infty$ and $0<q\leq \infty$
	\begin{equation*}
	T_a^\varphi: B_{p,q}^{s+m-m_c(p)}(\R^n)\rightarrow B_{p,q}^s(\R^n).
	\end{equation*}
	
\noindent In particular taking $m=m_c(p)$ in both cases, we have that
	\begin{equation*}
	T_a^\varphi: B_{p,q}^{s}(\R^n)\rightarrow B_{p,q}^s(\R^n).
	\end{equation*}
\end{enumerate}

\end{thm}
\begin{proof}
Once again we split $T_a^\varphi$ into a low and a high frequency part. Indeed, take $\psi_0$ as in Definition \ref{def:LP}, i.e.
	\begin{equation*}
    \begin{split}
	T_a^\varphi f(x) &=\int_{\R^n} \psi_0(\xi) a(x,\xi)  e^{i\varphi(x,\xi)} \widehat{f}(\xi)\ddd \xi \\ &+ \int_{\R^n} (1-\psi_0(\xi)) a(x,\xi)  e^{i\varphi(x,\xi)} \widehat{f}(\xi)\ddd \xi  \\ &=: T_1 f (x) + T_2 f(x).
    \end{split}
	\end{equation*}
Now for (i) we use Proposition \ref{prop:low1 BL} and for (ii) Proposition \ref{prop:low2 BL} (taking \linebreak $s_1=s+m-m_c(p)$, $s_2=s$ and $q_1=q_2=q$). These yield that \linebreak $T_1: B_{p,q}^{s+m-m_c(p)}(\R^n)\rightarrow B_{p,q}^s(\R^n).$ For $T_2f$ Proposition \ref{prop:high} yields that \linebreak $T_2: B_{p,q}^{s+m-m_c(p)}(\R^n)\rightarrow B_{p,q}^s(\R^n).$
\end{proof}\hspace*{1cm}

\section{Boundedness of FIOs on Triebel-Lizorkin spaces}\label{boundedness on Triebel}

In this section we investigate the boundedness of FIO's on Triebel-Lizorkin spaces spaces $F^s_{p,q}(\R^n)$ for $-\infty <s<\infty,$ $0<p\leq \infty$ and $0<q\leq \infty.$ Some of the results that we drive are based on the Besov-Lipschitz results which were obtained in the previous sections, a couple are obtained by interpolation, and some through direct methods. Once again, both local and global cases will be treated here.
We start with the following result which is sharp, up to the end point.

\begin{thm}\label{thm:local and global nonendpoint TL}
Let $a (x,\xi)\in S^{m}(\R^n)$ and $\displaystyle m_c(p) := -(n-1)\abs {\frac 1p -\frac 12}.$ Assume also that $\varphi(x,\xi) \in\mathcal{C}^{\infty}(\R^n \times\R^n \setminus\{0\}),$ is positively homogeneous of degree one in $\xi$. If $m<m_c(p)$, then in either of the following cases, we have that $T_a^\varphi$ is bounded from $F_{p,q}^{s}(\R^n)$ to $F_{p,q}^s(\R^n).$
\begin{enumerate}
\item[\emph{(i)}] $s\in \R$, $0<p< \infty$ and $0<q\leq \infty$\emph{;}  $a(x,\xi)$ has compact support in $x$, and $\varphi$ is non-degenerate on the support of $a(x,\xi),$

\item[\emph{(ii)}] $s\in \R$, $\displaystyle \frac{n}{n+1}<p< \infty$, $0<q\leq \infty,$ $\varphi (x,\xi)\in \Phi^2$ is $\mathrm{SND}.$
\end{enumerate}

\end{thm}
\begin{proof}
	Take $\varepsilon>0$ in such a way that $m+2\varepsilon \leq m_c(p)$. Then using the embedding \eqref{embedding of TL}, equality \eqref{equality of TL and BL}, Theorem \ref{thm:local and global BL}, and finally the fact that $(1-\Delta)^{-\frac \varepsilon 2}$ is an isomorphism from $F_{p,q}^{s}(\R^n)$  to $F_{p,q}^{s+\varepsilon}(\R^n)$,
we have that
\m{
\Vert T_a^\varphi f\Vert_{F^{s}_{p,q}(\R^n)}&=\Vert T_a^\varphi (1-\Delta)^{\varepsilon}(1-\Delta)^{-\varepsilon}f\Vert_{F^s_{p,q}(\R^n)}\\ &\lesssim \Vert T_a^\varphi (1-\Delta)^{\varepsilon}(1-\Delta)^{-\varepsilon}f\Vert_{F^{s+\varepsilon}_{p,p}(\R^n)}\lesssim \Vert (1-\Delta)^{-\varepsilon}f\Vert_{F^{s+\varepsilon}_{p,p}(\R^n)} \\ & \lesssim \Vert (1-\Delta)^{-\varepsilon}f\Vert_{F^{s+2\varepsilon}_{p,q}(\R^n)}\lesssim \Vert f\Vert_{F^{s}_{p,q}(\R^n)}.
}
\end{proof}
But indeed this result can be extended to the endpoint $m=m_c(p)$ if $q=2$, at least for $0<p\leq 2$. This, in the local case, i.e. the case of amplitudes with compact spatial support $p$ could be taken in the interval $(0,\infty]$. However, with the conditions of Theorem \ref{thm:local and global nonendpoint TL} above, one can prove a global version of the boundedness of FIOs on $F^s_{p,2}$, whose proof is based on the techniques developed by Seeger-Sogge-Stein \cite{SSS} and M. Ruzhansky and M. Sugimoto \cite{Ruzhansky-Sugimoto}. The long and rather technical proof will occupy the next subsection.\\

\subsection{Triebel-Lizorkin boundedness of the high frequency portion of FIOs}\label{section:the local section_high_Triebel}
First we consider the boundedness of FIOs with high frequency amplitudes on Triebel-Lizorkin spaces $F_{p,2}^0(\R^n)$ for $0<p\leq 1$. As was mentioned in Definition \ref{def:Triebel}, $F_{p,2}^{0}(\R^n)= h^p(\R^n)$ is the local Hardy space of Goldberg \cite{Goldberg}, and we shall  use the atomic decomposition of these spaces in order to carry out our agenda. The idea behind the proof of the following proposition was contained in an unpublished manuscript of the second and the third authors of this paper and D. Rule \cite{multilinearfio}, which dealt with FIOs with phase functions of the form $\phi(\xi)+ x\cdot \xi$ where $\phi\in \mathcal{C}^{\infty}(\mathbb{R}^n \setminus 0)$ is positively homogeneous of degree $1$. In this paper we have generalised that result to the case of SND phase functions which belong to $\Phi^2$.   

\begin{prop}\label{Thm:high TL} Let $\psi_0$ be as cut-off function as in \emph{Definition \ref{def:LP}}, $p\in (0,\infty]$ and $m_c(p)$ the critical order defined in \eqref{eq:criticaldecay}. Assume that  $a\in S^{m_c(p)}(\R^n)$ and $\varphi\in \Phi^2$ is a phase function that verifies the $\mathrm{SND}$ condition \eqref{SND}. Then for $n\geq 2$ the operator $T_{\sigma}^\varphi$ given by \eqref{definition of FIO} with amplitude $\sigma( x, \xi):= (1-\psi_0(\xi))\,a(x,\xi)$ 
satisfies $F_{p,2}^{s}(\R^n)\rightarrow F_{p,2}^s(\R^n),$ for $-\infty<s<\infty$.
\end{prop}
\begin{proof}
We divide the proof in different steps as follows.\\
\begin{enumerate}
\item In Step 1 we consider the case when $s=0$, $0<p\leq 1$, $n\geq 2$ and an $h^p$ atom $a$ with support inside a ball with radius $r\leq 1$. We also assume that the amplitude is compactly supported in the $x$-variable. We show that for, $\Vert T_\sigma^\varphi a\Vert_{L^p(\R^n)} \leq C,$ where the constant $C$ doesn't depend of $a$ and $r$.
\item In Step 2 we assume the same premises as in Step 1 with the only difference that $r\geq1$. Step 1 and 2 will together imply that $T_\sigma^\varphi : h^p(\R^n)\to L^p(\R^n)$ for $0<p\leq 1$.
\item In step 3 we prove that $\Vert (T_\sigma^\varphi)^* a\Vert_{L^1(\R^n)} \leq C,$ for $h^1$ atoms $a$ supported in balls of radii $r<1$.
\item In Step 4 we assume the same premises as in Step 3 with the only difference that $r\geq1$. Step 3 and 4 will together imply that $(T_\sigma^\varphi)^* : h^1(\R^n)\to L^1(\R^n)$.
\item In Step 5 we globalise the local results obtained in the previous steps.
\item In Step 6 we lift the results to $T_\sigma^\varphi :h^p(\R^n) \to h^p(\R^n)$ and \linebreak$(T_\sigma^\varphi)^\ast:h^1(\R^n) \to h^1(\R^n).$
\item We conclude the proof by showing the boundedness of $T_\sigma^\varphi$ on $F_{p,2}^{s}(\R^n)$ for $0<p\leq \infty$ and $-\infty<s<\infty $. 
\end{enumerate}\hspace*{1cm}\\
These steps will conclude the proof.\\

{\bf{Step 1 -- Estimates of $\mathbf{\norm{T_\sigma^\varphi a}_{L^p(\R^n)}}$ when $\mathbf{r\leq 1}$}}

In what follows, let $a(x)$ be an $h^p$ atom supported in a ball $B$ of radius $r\leq 1$ and centre $\overline{y}$. Now split
\begin{equation}\label{eq: spliting the FIO on balls}
  \int_{\mathbb{R}^{n}}|T_\sigma^\varphi a(x)|^{p}\dd x=\int_{B^{*}}|T_\sigma^\varphi a(x)|^{p}\dd x+\int_{B^{*c}}|T_\sigma^\varphi a(x)|^{p}\dd x,
\end{equation}

where $B^{*}$ is as in Definition \ref{def:influenceset}.  H\"older's inequality and the observation in Remark \ref{rem:influenceset} (i) yield that
\begin{equation*}
\int_{B^{*}}|T_\sigma^\varphi a(x)|^{p}\dd x\leq|B^{*}|^{1-\frac p2}\brkt{\int_{B^{*}}|T_\sigma^\varphi  a(x)|^{2}\dd x}^{\frac p2}
\lesssim r^{\brkt{1-\frac p2}}\Vert T_\sigma^\varphi  a\Vert_{L^{2}(\R^n)}^{p}.
\end{equation*}
To analyse the second term on the right hand side of \eqref{eq: spliting the FIO on balls} we proceed as follows.\\

First assume that $\displaystyle -\frac{n}{2}<m_c(p)<0$. Then there exists a $1<q<2$ such that $\displaystyle \frac{1}{2}=\frac{1}{q}+\frac{m_c(p)}{n}$. Observe that it is at this point where the assumption on the dimension $n$ plays a role. Indeed the case $n=1$ cannot satisfy this assumption, as then $m_c(p)=0$. Using the global $L^2 \to L^2$ boundedness of the operator $T_\sigma^\varphi $ and the estimates for Riesz potentials we can deduce that $T_\sigma^\varphi $ is bounded from $L^{q}(\mathbb{R}^{n})$ to $L^{2}(\mathbb{R}^{n})$  and therefore
$$
\Vert T_\sigma^\varphi  a\Vert_{L^{2}(\R^n)}^{p}\lesssim\Vert a\Vert_{L^{q}(\R^n)}^{p}\lesssim |B|^{\frac pq -1}\lesssim r^{n\brkt{\frac pq -1}}.
$$
Thus
$$
\int_{B^{*}}| T_\sigma^\varphi a(x)|^{p}\dd x\lesssim r^{1-\frac p2+n\brkt{\frac pq -1}}.
$$
To see that $\displaystyle r^{1-\frac p2+n\brkt{\frac pq -1}}= 1$, we observe that since $\displaystyle \frac{p}{q}= \frac{p}{2}-\frac{p\, m_c(p)}{n}$ we have
\m{&1-\frac{p}{2}+n\brkt{\frac{p}{q} -1} = 1-\frac{p}{2}+n\brkt{\frac{p}{2}-\frac{p\, m_c(p)}{n}-1}\\ =\; & p\brkt{\frac{1}{p}-\frac{1}{2}+\frac{n}{2}- m_c(p)-\frac{n}{p}}= p\brkt{-(n-1)\brkt{\frac{1}{p}-\frac{1}{2}}- m_c(p) }.}
Now since $\displaystyle m_c(p)= -(n-1)\brkt{\frac{1}{p}-\frac{1}{2}}$, it follows that $\displaystyle r^{1-\frac p2 +n\brkt{\frac pq -1}}= 1$.

If instead $\displaystyle m_c(p)\leq -\frac{n}{2},$ then setting $\displaystyle b=|B|^{\frac 1p- \frac 1q}a$, with $\displaystyle \frac{1}{2}=\frac{1}{q}+\frac{m_c(p)}{n}$ as before (so now $q < p < 1$) we see that $b$ is an $h^q$-atom with the same support as $a$. In fact $b$ becomes an atom in the Hardy space $H^{q}(\mathbb{R}^{n})$, so by the results in \cite{Krantz}*{ Corollary 2.3}, we have that $T_\sigma^\varphi :H^{q}(\mathbb{R}^{n})\rightarrow L^{2}(\mathbb{R}^{n})$ is bounded and
\m{
\int_{B^{*}}|T_\sigma^\varphi  a(x)|^{p}\dd x&\lesssim r^{1-\frac p2}\Vert a\Vert_{H^{q}(\R^n)}^{p}
\lesssim r^{1-\frac p2}|B|^{p\brkt{\frac 1q-\frac 1p}}\Vert b \Vert_{H^{q}(\R^n)}^{p}
\\ &\lesssim r^{1-\frac p2+n\brkt{\frac pq-1}} = 1.
}

Using the partition of unity that was introduced in \eqref{PL partition of unity}
we can write
\begin{equation}\label{defn of sss pieces of T}
 T_\sigma^\varphi  =\sum_{j=0}^{\infty}T_{j}=:\sum_{j}\sum_{\nu}T_{j}^{\nu}.
\end{equation}

Now to deal with the integral $\displaystyle \int_{B^{*c}}|T_\sigma^\varphi a(x)|^{p}\dd x,$ using the notation in \eqref{defn of sss pieces of T}, we observe that
\begin{equation}\label{Tphisigma2 over BC}
  \int_{B^{*c}}|T_\sigma^\varphi  a(x)|^p \, \dd x\leq \sum_{2^{j}<r^{-1}}\int_{B^{*c}}|T_{j}a(x)|^p \, \dd x +\sum_{2^{j}\geq r^{-1}}\int_{B^{*c}}|T_{j}a(x)|^p \,\dd x.
\end{equation}
For $2^{j}<r^{-1} $ we use \eqref{eq:operatorestimate1} to deduce
\nm{grunka 2}{
&\int_{B^{*c}}|T_{j}a(x)|^p \dd x \\ \lesssim & \sum_{\nu} {  \int_{\R^n}\frac{ 2^{j\brkt{m_c(p)+\frac{n+1}2}p}2^{j Mp}r^{ Mp}r^{np-n}   }{\brkt{1+2^{2j}\abs{\nabla_{\xi_1}\varphi(x,\xi_j^\nu)-\bar y_1}^2}^{Np}\brkt{1+2^{j}\abs{\nabla_{\xi'}\varphi(x,\xi_j^\nu)-\bar y'}^2}^{Np}}} \dd x \\ \lesssim & \sum_{\nu} {  \int_{\R^n}\frac{ 2^{j\brkt{m_c(p)+\frac{n+1}2}p}2^{j Mp}2^{-j\frac{n+1}2}r^{ Mp}r^{np-n}   }{\brkt{1+\abs{x_1}^2}^{Np}\brkt{1+\abs{x'}^2}^{Np}}} \dd x \\ \lesssim &\, 2^{j\frac{n-1}{2}} 2^{jp\brkt{m_c(p)+\frac{n+1}2}}2^{j Mp}2^{-j\frac{n+1}2}r^{ Mp}r^{np-n}
}
where we have used the non-degeneracy condition on $\varphi$ to make the change of variables and the fact the are $O\brkt{2^{j\frac{n-1}{2}}}$ terms in the sum over $\nu$. Summing over $2^{j}<r^{-1} $ yields 
\m{\sum_{2^{j}<r^{-1}}\int_{B^{*c}}|T_{j}a(x)|^p \dd x \lesssim 1}

if $M$ and $N$ are chosen appropriately. For the second term in \eqref{Tphisigma2 over BC} we use \eqref{eq:operatorestimate2} and $2^{j}\geq r^{-1}$ to deduce
\nm{grunka 1}{
&\int_{B^{*c}}|T_{j}a(x)|^p \dd x \\&\lesssim \sum_\nu\int_{\R^n}{  \frac{2^{jp\brkt{m_c(p)+\frac {n+1}2}}  2^{-jMp }r^{-Mp}r^{np-n} 
2^{4jNp}r^{4Np}
}{\brkt{1+2^{2j}\abs{\nabla_{\xi_1}\varphi(x,\xi_j^\nu)-\bar y_1}^2}^{Np}\brkt{1+2^{j}\abs{\nabla_{\xi'}\varphi(x,\xi_j^\nu)-\bar y'}^2}^{Np}}}  \dd x \\&\lesssim \sum_\nu\int_{\R^n}{  \frac{2^{jp\brkt{m_c(p)+\frac {n+1}2}}  2^{-jMp }2^{-\frac{n+1}{2}}r^{-Mp}r^{np-n} 
2^{4jNp}r^{4Np}
}{\brkt{1+\abs{x_1}^2}^{Np}\brkt{1+\abs{x'}^2}^{Np}}}  \dd x \\&\lesssim   2^{j\frac{n-1}{2}}2^{jp\brkt{m_c(p)+\frac {n+1}2}}  2^{-jMp }2^{-j\frac{n+1}{2}}r^{-Mp}r^{np-n} 
2^{4jNp}r^{4Np}
}
where we once again have used the non-degeneracy condition on $\varphi$ to make the change of variables and the fact the are $O\brkt{2^{j\frac{n-1}{2}}}$ terms in the sum over $\nu$. Summing over $2^{j}\geq r^{-1} $ yields 
\m{\sum_{2^{j}\geq r^{-1}}\int_{B^{*c}}|T_{j}a(x)|^p \dd x \lesssim 1}

for appropriate $M$ and $N$. This proves \eqref{eq: spliting the FIO on balls} for balls of radius less than or equal to one.\\

{\bf{Step 2 -- Estimates of $\mathbf{\norm{T_\sigma^\varphi a}_{L^p(\R^n)}}$ when $\mathbf{r> 1}$}}

This part of the estimate can be easily handled by the H\"older inequality. Indeed we have
\m{
\int_{\R^n} |T_\sigma^\varphi  a(x)|^p\, \dd x &\leq |\supp_x \sigma|^{1-\frac p2} \brkt{\int_{\R^n} |T_\sigma^\varphi  a(x)|^2\, \dd x } ^{\frac p2}\lesssim \Vert a\Vert^{p}_{L^2(\R^n)}\\& \lesssim r^{n\brkt{\frac p2-1}}\lesssim 1,
}
as desired.\\

{\bf{Step 3 -- Estimates of $\mathbf{\norm{(T_\sigma^\varphi)^* a}_{L^p(\R^n)}}$ when $\mathbf{r\leq 1}$}} \\
As in Step 1 we split 
\begin{equation}\label{eq: spliting the FIO on balls2}
  \int_{\mathbb{R}^{n}}|(T_\sigma^\varphi)^* a(x)|\dd x=\int_{B^{*}}|(T_\sigma^\varphi)^* a(x)|\dd x+\int_{B^{*c}}|(T_\sigma^\varphi)^*a(x)|\dd x
\end{equation}
and for the first term we proceed exactly as in Step 1 when $p=1$. For the second term we once again use the partition of unity that was introduced in Definition \ref{def:LP} and write 
\begin{equation}\label{defn of sss pieces of T2}
 (T_\sigma^\varphi)^* =\sum_{j=0}^{\infty}T_{j} =\sum_{2^j<r^{-1}}^{\infty}T_{j} +\sum_{2^j\geq r^{-1}}^{\infty}T_{j}.
\end{equation}
Starting with the case $2^j<r^{-1}$, Lemma \ref{SSS estimates for Kj2} (i), moment condition on $a$ and Minkowski's inequality yield that
\nm{eq:SSS adjoint small ball estimate}{
\norm{T_ja}_{L^1(B^{*c}) } &\lesssim \int_{B^{*c} } \int_B \abs{K_j(x,y)-K_j(x,\bar y)}\abs{a(y)} \dd y \dd x \\&\lesssim  \int_B \abs{a(y)}\brkt{\int_{B^{*c} }  \abs{K_j(x,y)-K_j(x,\bar y)}  \dd x} \dd y \\ & \lesssim 2^j\, r,
}
where $\overline{y}$ is as usual the centre of the support of the atom $a$. Summing in $j$ we obtain 
\m{
\sum_{2^j<r^{-1}}\norm{T_ja}_{L^1(B^{*c}) } \lesssim 1
}
Now we turn into the case when $2^j<r^{-1}$. Lemma \ref{SSS estimates for Kj2} (ii) and Minkowski's inequality yields 
\m{
\norm{T_ja}_{L^1(B^{*c}) } &\lesssim \int_{B^{*c} } \int_B \abs{K_j(x,y)}\abs{a(y)} \dd y \dd x \\&\lesssim  \int_B \abs{a(y)}\brkt{\int_{B^{*c} }  \abs{K_j(x,y)}  \dd x} \dd y \\ & \lesssim 2^{-j}\, r^{-1}
}
and summing in $j$ we get that 
\m{
\sum_{2^j\geq r^{-1}}\norm{T_ja}_{L^1(B^{*c}) } \lesssim 1
}
and hence the desired results of Step 3 has been proven.\\

{\bf{Step 4 -- Estimates of $\mathbf{\norm{(T_\sigma^\varphi)^* a}_{L^p(\R^n)}}$ when $\mathbf{r> 1}$}}\\
Now we go for atoms with its support in balls of radii $r\geq 1$. To prove the boundedness of the adjoint $\brkt{T_a^\varphi}^*$, we split the $L^1$-norm into two pieces, namely 
\begin{equation}\label{eq: splitting the FIO on balls3}
  \int_{\mathbb{R}^{n}}|(T_\sigma^\varphi)^* a(x)|\dd x=\int_{B'}|(T_\sigma^\varphi)^* a(x)|\dd x+\int_{B'^{c}}|(T_\sigma^\varphi)^*a(x)|\dd x
\end{equation}
where $B'$ is a ball centered at the origin and has a radius of $2K$, where $K:= \sup_{(y,\xi)\in (\supp \sigma\,\cap\, (\R^n \times \mathbb{S}^{n-1}))}|\nabla_\xi \varphi(y,\xi)|$.  We treat the first term of \eqref{eq: splitting the FIO on balls3} as in Step 2. For the second term we observe that the kernel of 
$(T_\sigma^\varphi)^*$ satisfies 
\begin{equation}\label{based on nonstationary}
\abs{\int_{\R^n} e^{ix\cdot\xi -i\varphi(y,\xi)} \sigma(y,\xi)\ddd \xi } \lesssim \frac 1{\abs x^{N}},
\end{equation}
for $|x|>2K$. This follows from the fact that the modulus of the gradient of the phase of the oscillatory integral in \eqref{based on nonstationary} satisfies $|x-\nabla_\xi\varphi(y,\xi)|\geq |x|-K \geq |x|/2$. Now a standard non-stationary phase argument yields \eqref{based on nonstationary}. Hence 
\m{
\int_{B'^{c}}|(T_\sigma^\varphi)^*a(x)|\dd x \lesssim \int_{B'^{c}}\frac 1{\abs x^{N}}\brkt{\int_B \abs{a(y)}\dd y} \dd x \lesssim 1.
}\hspace*{1cm}\\

{\bf{Step 5 --  Globalisation of the local result}} \\
We will proceed by globalising the previous result for both $T_\sigma^\varphi$ and $\brkt{T_\sigma^\varphi}^* $ at the same time. For the latter we will only consider the case when $p=1$. Whenever we write $T$ we refer to both $T_\sigma^\varphi$ and $\brkt{T_\sigma^\varphi}^* $.\\

To prove that $\displaystyle \int_{\mathbb{R}^{n}}|Ta(x)|^{p}\dd x\lesssim 1$ when there is no requirement on the support of the amplitude, we need to use a different strategy. First we observe that a global norm estimate for $Ta$ with $a$ supported in a ball with an arbitrary centre, would follow from a uniform in $s$ norm-estimate for $\tau_s^* T\tau_s a$, with an atom $a$ whose support is inside a ball centred at the origin. This is because by translation invariance of the $L^p$ norm one has that $\Vert Ta\Vert_{L^p(\R^n)}= \Vert \tau_s^* T\tau_s \tau_{-s} a\Vert_{L^p(\R^n)}.$ Note that here $\tau_s$ is the operator of translation by $s\in\R^n$. Thus our goal is to establish that $\Vert \tau_s^* T\tau_s a\Vert_{L^p(\R^n)}\lesssim 1$, where the estimate is uniform in $s$ and $a$ has the support in a ball centred at the origin.\\

At this point we once again use the conditions on the phase function and Theorem \ref{invariance thm} on the invariance of Triebel-Lizorkin spaces under diffeomorphisms as in the proof of Lemma \ref{prop:low2 BL},  to reduce our analysis to the case of operators with $\varphi$ of \smallskip the form $\theta (x,\xi)+x\cdot \xi$ or $-\theta (y,\xi)-y\cdot \xi$ with $\theta \in \Phi^1$. Now let $r\geq1$, $\displaystyle L>\frac{n}{p}$ and $s\in\R^n$.
Suppose $a$ is an $h^p$ atom supported in a ball $B$, centred at the origin, with radius $r$. Split the $L^p$ norm of $\tau_s^* T\tau_s a$ into following two pieces:
\[
\left\| \tau_s^* T\tau_s a \right\|_{L^p(\R^n)}\leq \left\|\tau_s^* T\tau_s a\right\|_{L^p\brkt{\widetilde{\Delta}_{2r}}}+\left\|\tau_s^* T^\phase_{\sigma}\tau_s a\right\|_{L^p\brkt{\R^n \setminus \widetilde{\Delta}_{2r}}}
\]
First let us show that
\[
\left\|\tau_s^* T\tau_s a\right\|_{L^p\brkt{\widetilde{\Delta}_{2r}}}
\leq  C(n,M_{L},N_{L+1}).
\]

For $x\in \widetilde{\Delta}_{2r}$ and $|y|\leq r$, we have
$\widetilde{H}(x)\leq2 H(x,y,x-y)$ and $(x,y,x-y)\in\Delta_r$
by Lemma \ref{Lem:outside}.
This fact and Lemma \ref{Lem:kernel RuzhSug} yield for any atom $a$ supported in $B(0,r)$
\begin{align}\label{Ruz-Sug ball estim}
| T a(x)|&\leq 2^{L}\widetilde{H}(x)^{-L}\int_{|y|\leq r}\abs{H(x,y, x-y)^{L}K(x,y,x-y)\, a(y)} \dd y
\\\nonumber &\leq 2^{L}\widetilde{H}(x)^{-L}
\norm{ H^{L} K}_{L^\infty( \Delta_r)}
\Vert a\Vert_{L^1(\R^n)}\\\nonumber  &\leq C(n,L,M_L ,N_{L+1})\,  \widetilde{H}(x)^{-L}\, r^{n\brkt{1-\frac 1p}},
\end{align}
since $\displaystyle \Vert a\Vert_{L^1(\R^n)} \leq |B|^{1-\frac 1p}.$ 
Therefore, if $r\geq 1$, choosing $\displaystyle L>\frac np$, Lemma \ref{Lem:kernel RuzhSug} and the monotonicity
of $\Delta_r$ yield
\begin{align}\label{grunka 3}
\left\|T a\right\|_{L^p\brkt{\widetilde{\Delta}_{2r}}}
&\lesssim
\left\|\widetilde{H}(x)^{-L}\right\|_{L^p\brkt{\widetilde{\Delta}_{2r}}}
\leq C(n, M_L, N_{L+1}).
\end{align}
Observe that the phase function and the amplitude of $\tau_s^* T\tau_s$ are of the form  {$ \theta (x+ \nolinebreak  s, \xi)+(x-y)\cdot\xi$ and  $\sigma (x+s, \xi)$} respectively when $T= T_\sigma^\varphi$ (a similar property is also true for $\brkt{T_\sigma^\varphi}^*$). Therefore the conjugation of $T$ by $\tau_s$ renders the constants $M_L$ and $N_{L+1}$ unchanged and therefore the estimate above also yields the very same one for $\tau_s^* T\tau_s$. This means that
$\left\|\tau_s^* T\tau_s a\right\|_{L^p\brkt{\widetilde{\Delta}_{2r}}}\lesssim 1$.\\

On the other hand for $\left\|\tau_s^* T\tau_s a\right\|_{L^p\brkt{\R^n \setminus \widetilde{\Delta}_{2r}}},$ Lemma \ref{Lem:outside}, H\"older's inequality and the properties of the atom $a$ yield that
\begin{align*}
\left\|\tau_s^* T\tau_s a\right\|_{L^p\brkt{\R^n \setminus \widetilde{\Delta}_{2r}}}
&\leq
\abs{\R^n \setminus \widetilde{\Delta}_{2r}}^{1-\frac p2}
\left\|\tau_s^* T\tau_s a \right\|^{p}_{L^2(\R^n)}
\\
&\lesssim
r^{n\brkt{1-\frac p2}}\|a\|^{p}_{L^2(\R^n)}
\lesssim r^{n\brkt{1-\frac p2}} r^{n\brkt{\frac p2 -1}} = 1.
\end{align*}
Now if the atom is supported in a ball of radius $r\leq 1$ then clearly $\supp a\subset B(0,1).$ Now write $\R^n = \widetilde{\Delta}_2 \cup (\R^n \setminus \widetilde{\Delta}_2)$ and observe that we can now use Lemma \ref{Lem:kernel RuzhSug} with $r=1$ to conclude that 
\begin{equation*}
|Ta(x)| \lesssim\widetilde{H}(x)^{-L},
\end{equation*}
which in turn yields that $\Vert \tau_s^* T\tau_s a\Vert_{L^p \brkt{\widetilde{\Delta}_2}}\lesssim 1.$
Using now the first part of Lemma \ref{Lem:outside} we see that $\R^n \setminus \widetilde{\Delta}_2 \subset B(0, 2+N_{K})$ which together with the local boundedness result that we established previously implies that
\begin{equation}
\Vert \tau_s^* T\tau_s a\Vert_{L^p \brkt{\R^n \setminus\widetilde{\Delta}_2}}\lesssim \Vert \tau_s^* T\tau_s a\Vert_{L^p (B(0, 2+N_{K}))}\lesssim \Vert a\Vert_{h^p (\R^n)}\lesssim 1.
\end{equation}\\

{\bf{Step 6 -- Lifting the result to $\mathbf{h^p\to h^p} $ boundedness }} \\
In order to boost up the $h^p\to L^p$ boundedness of $T^\varphi_a$ to the desired $h^p\to h^p$ boundedness, we follow the strategy in \cite{PelosoSecco}. As it was shown in that paper, in order to show that $f\in h^p(\R^n)$ it is enough to prove that $r_{\varepsilon}^{\alpha}(D) f \in L^p(\R^n)$ uniformly in $\varepsilon$, where $\varepsilon\in (0,1]$, $\alpha=(\alpha_1, \dots,\, \alpha_n)\in \mathbb{Z}_{+}^n$, $ \displaystyle r_{\varepsilon}^{\alpha} (\xi)= \widehat{\Psi} (\varepsilon \xi) \prod_{i=1}^n \brkt{\frac{\xi_i}{|\xi|}}^{\alpha_i} \brkt{1-\widehat{\Theta}(\xi)}^{\alpha_i}$, $\Psi\in \mathcal{C}^\infty_c (\R^n)$ with $\displaystyle \int_{\R^n} \Psi (x)  \dd x =1$, and $\widehat{\Theta}$ is the smooth cut-off function  which is identically one in a
neighborhood of the origin. Moreover
$$\norm{ \widehat{\Theta}(D) f}_{L^p(\R^n)}+\sum_{M\leq |\alpha|\leq M+1} \sup_{0<\varepsilon\leq 1} \norm{r^{\alpha}_{\varepsilon} (D) f}_{L^p(\R^n)} \sim \Vert f\Vert _{h^p(\R^n)},$$
where $\displaystyle M=\left [n\brkt{\frac 1p -1} \right ]+1. $
As a consequence, it will be enough to prove that
\begin{equation}\label{eq:endpoints converse 1}
	\Vert r_{\varepsilon}^{\alpha}(D) T_\sigma^\varphi f\Vert_{L^p(\R^n)} \lesssim \Vert f\Vert_{h^{p}(\R^n)},
\end{equation}
uniformly in $\varepsilon$, and
\begin{equation}\label{eq:endpoints converse 2}
	\norm{ \widehat{\Theta}(D) T_\sigma^\varphi f }_{L^p(\R^n)} \lesssim \Vert f\Vert_{h^{p}(\R^n)}.
\end{equation}

To lift the boundedness for $T_\sigma^\varphi $ we proceed as follows:  If the phase function $\varphi\in \Phi^2$ and is SND, then using our phase reduction mentioned previously, it is not hard to show that for a reduced phase $\varphi$, one has $|\nabla_x \varphi(x,\xi)|\sim  |\xi|,$ and for $|\alpha|$, $|\beta|\geq 1$ we have $|\partial_{x}^{\alpha} \varphi(x,\xi)|\lesssim \langle \xi\rangle$, $|\partial_{\xi}^{\alpha}\partial_{x}^{\beta} \varphi(x,\xi)|\lesssim |\xi|^{1-|\alpha|}$ (observe also that $|\xi|$ is large). Moreover  $r_{\varepsilon}^{\alpha}(D)$ and $ \widehat{\Theta}(D)$ are pseudodifferential operators with symbols respectively in $S^0(\R^n)$ (uniformly in $\varepsilon$), and in $S^{-\infty}(\R^n)$. Therefore, using Theorem \ref{prop:monsteriosity} with $t=1$, we can see that the compositions $ r_{\varepsilon}^{\alpha}(D) T_\sigma^\varphi $ and $\widehat{\Theta}(D) T_\sigma^\varphi $ are FIOs with amplitudes in $S^{m_c(p)}(\R^n)$ and $S^{-\infty}(\R^n)$ and phase functions $\varphi$, and therefore \eqref{eq:endpoints converse 1} and \eqref{eq:endpoints converse 2} are both valid.\\

To lift the boundedness for $\brkt{T_\sigma^\varphi} ^\ast$ we proceed as follows: Observe that for any real-valued Fourier multiplier $p(\xi)\in S^s(\R^n)$ one has that $P(D)\brkt{T_\sigma^\varphi }^\ast=\brkt{T_\sigma^\varphi P(D)}^*$. Now $T_\sigma^\varphi P(D)$ is an FIO with the phase $\varphi$ and an amplitude in $S^{m_c(1)+s}(\R^n)$. In particular, if $ P(D)$ is either of $r_\varepsilon^\alpha(D)$ or $\widehat{\Theta}(D),$ then $T_\sigma^\varphi P(D)$ is an FIO with phase $\varphi$ and an amplitude in the class $S^{m_c(1)}(\R^n)$. However, in Step 3 and 4 of this proof, we have shown that the adjoints of such operators are bounded from $h^1(\R^n)$ to $ L^1(\R^n)$ and therefore $P(D)\brkt{T_\sigma^\varphi}^\ast$ is also bounded from $h^1(\R^n)$ to $  L^1(\R^n)$ and  once again \eqref{eq:endpoints converse 1} and \eqref{eq:endpoints converse 2} are valid uniformly in $\varepsilon\in (0,1]$. Therefore we have that $\brkt{T_\sigma^\varphi}^\ast$ is bounded on $h^1(\R^n)$.\\ 

{\bf{Step 7 -- Lifting to $\mathbf{F^s_{p,2}} $ }}\\
So far we have shown that $T_\sigma^\varphi$ is bounded from $F^0_{p,2}(\R^n)$ to itself, for $0<p<1$ and that $\brkt{T_\sigma^\varphi}^*$ is bounded on $F^0_{1,2}(\R^n)$. Hence using complex interpolation and thereafter duality, the operator $T_\sigma^\varphi$ is bounded from $F^0_{p,2}(\R^n)$ to itself for $0<p\leq \infty$. One uses a similar reasoning as in Theorem \ref{prop:monsteriosity} (or the global calculus of FIOs) to see that the operator $(1-\Delta)^{\frac s2} T^{\varphi}_a(1-\Delta)^{-\frac s2}$ is a similar operator associated to an amplitude in $S^{m_c(p)}(\R^n)$ and phase $\varphi$, and hence bounded from $F^0_{p,2}(\R^n)$ to itself. Therefore using the fact that the operator $(1-\Delta)^{\frac{s}{2}}$ is an isomorphism from $F^s_{p,2}(\R^n)$ to $F^0_{p,2}(\R^n)$ for $0<p\leq \infty$, we obtain the desired result of Proposition \ref{Thm:high TL}. 
\end{proof}

\subsection{Triebel-Lizorkin boundedness of the low frequency portion of FIOs}\label{section:the local section Triebel}
\noindent In this section we prove the boundedness of FIOs, where the amplitudes are frequency-supported in a neighbourhood of the origin. This is quite similar to the case of Besov-Lipschitz spaces and we shall use the estimates that were developed in that context. As before, $T_{a_0}^\varphi$ will denote the FIO with amplitude $a_0(x,\xi)=a(x,\xi)\psi_0(\xi).$

where $\psi_0$ is as in Definition \ref{def:LP}. We start with the local result:
\begin{prop}[Local boundedness]\label{prop:low1 TL}
Let $a (x,\xi)\in S^{m}(\R^n)$, with compact support in $x$ and $\varphi (x,\xi)\in \mathcal{C}^{\infty}(\R^n \times \R^n\setminus\{0\})$ be positively homogeneous of degree one in $\xi$ and non-degenerate on the support of $a(x,\xi)$. Then $T_{a_0}^\varphi  :F_{p,q_1}^{s_1}(\R^n)\to F_{p,q_2}^{s_2}(\R^n)$, for $0<p,q_1,q_2 \leq\infty$, $-\infty <s_1,s_2<\infty$.
\end{prop}

\begin{proof}
Using \eqref{jaevla namn 2}, for $0<p\leq \infty$  we have the pointwise estimate
\m{
\left |\psi_j(D)T_{a_0}^\varphi  f(x) \right |
\lesssim 2^{j(n-2N_1)} \norm{f}_{L^p(\R^n)},
}

from which it follows that
\begin{equation*}
\begin{split}
\norm{T_{a_0}^\varphi  f}_{F_{p,q_2}^{s_2}(\R^n)} &= \norm{  \brkt{\sum_{j=0}^\infty 2^{js_2q_2}|\psi_j(D)T_{a_0}^\varphi f|^{q_2}} ^{\frac 1{q_2}}  }_{L^p(\R^n)}\\ &\lesssim \left (\sum_{j=0}^\infty 2^{jq_2(s_2+n-2N_1)} \norm{f}_{L^{p}(\R^n)}^{q_2} \right )^{\frac 1{q_2}} \\ & =  \norm{f}_{L^{p}(\R^n)}\left (\sum_{j=0}^\infty 2^{jq_2(s_2+n-2N_1)} \right )^{\frac 1{q_2}}\\& \lesssim \norm{ f}_{L^{p}(\R^n)}\lesssim \norm{ f}_{F_{p,q_1}^{s_1}(\R^n)}.
\end{split}
\end{equation*}
\end{proof}
Now we state and prove the global boundedness of FIOs with frequency localised amplitudes on Triebel-Lizorkin spaces. The proof of this is similar to that of Propositions \ref{prop:low2 BL}, and \ref{prop:low1 TL} and hence is omitted.
\begin{prop}[Global boundedness]\label{prop:low2 TL}
Let $a(x,\xi)\in S^m(\R^n)$ and $\varphi (x,\xi)\in \Phi^2$ verifies the $\mathrm{SND}$ condition. Then  $T_{a_0}^\varphi :F_{p,q_1}^{s_1}(\R^n)\to F_{p,q_2}^{s_2}(\R^n)$, for $\displaystyle \frac{n}{n+1}<p\leq\infty$, $0<q_1,q_2 \leq\infty$, $-\infty <s_1,s_2<\infty$.
\end{prop}\hspace*{1cm}
\subsection{Local and Global boundedness of FIOs on Triebel-Lizorkin spaces}\label{section.local and global triebel for fios}
\noindent In this section we state and prove the local and global boundedness of Fourier integral operators on Triebel-Lizorkin spaces. In view of the results of the previous sections, what remains to do is to put the high and low frequency results for various cases together.\\

\begin{thm}\label{thm:local and global TL}
Let $p\in (0,\infty]$, $a (x,\xi)\in S^{m_c (p)}(\R^n)$ and $\varphi(x,\xi) \in\mathcal{C}^{\infty}(\R^n \times\R^n \setminus\{0\}),$ be positively homogeneous of degree one in $\xi$. Then under these assumptions, the following results hold true$:$

\begin{enumerate}
\item[\emph{(i)}] If $a(x,\xi)$ has compact support in $x$ and $\varphi$ is non-degenerate on the support of $a(x,\xi),$
then for any $s\in \R$ and $0<p\leq \infty$ the operator $T_a^\varphi $
is bounded from $F_{p,2}^{s}(\R^n) $ to $ F_{p,2}^s(\R^n)$.

\item[\emph{(ii)}] If $\varphi (x,\xi)\in \Phi^2$ is $\mathrm{SND}$, then for any $s\in \R$ and $\displaystyle \frac{n}{n+1}<p\leq \infty$, the operator $T_a^\varphi$ is bounded from $F_{p,2}^{s}(\R^n) $ to $ F_{p,2}^s(\R^n)$.
\end{enumerate}

\end{thm}
\begin{proof}
For the proof of (i), one observes that the compact support, the homogeneity and the non-degeneracy of the phase function yield that $$\abs{\det\brkt{\partial^2_{x_j\xi_k}\varphi (x,\xi)}}\geq \min_{(x,\xi)\, \in\, \supp\, a \times \mathbb{S}^{n-1}}\abs{\det\brkt{\partial^2_{x_j\xi_k}\varphi (x,\xi)}}>0.$$ Moreover, the same conditions on the phase also yield that $\varphi\in \Phi^2$. Thus for the high frequency portion of the FIO, the desired boundedness follows from the same arguments as in the proof of Proposition \ref{Thm:high TL}. Now adding the low frequency result of Proposition \ref{prop:low1 TL}, we can conclude the proof of (i).\\

To prove (ii) one just combines the results of Proposition \ref{Thm:high TL} and Proposition \ref{prop:low2 TL}.\\

For the case of $n=1$, we can split $T_a^\varphi$ into two pseudodifferential operators and a smoothing operator. For the details see the proof of Theorem \ref{dimension_one}.
\end{proof}

\section{Results obtained by interpolation}
As was mentioned before, using our results concerning Besov-Lipschitz and Triebel-Lizorkin boundedness of FIOs, we can also extend the ranges of Triebel-Lizorkin boundedness a bit further. This is done by complex interpolation (see e.g. \cite{Kalton}) in the vertical direction between $F^s_{p,p}(\R^n)=B^s_{p,p}(\R^n)$ and $F^s_{p,2}(\R^n)$ (as in Figure 1).

\vspace*{.5cm}
\begin{figure}[ht!]
\centering\includegraphics[scale=.75]{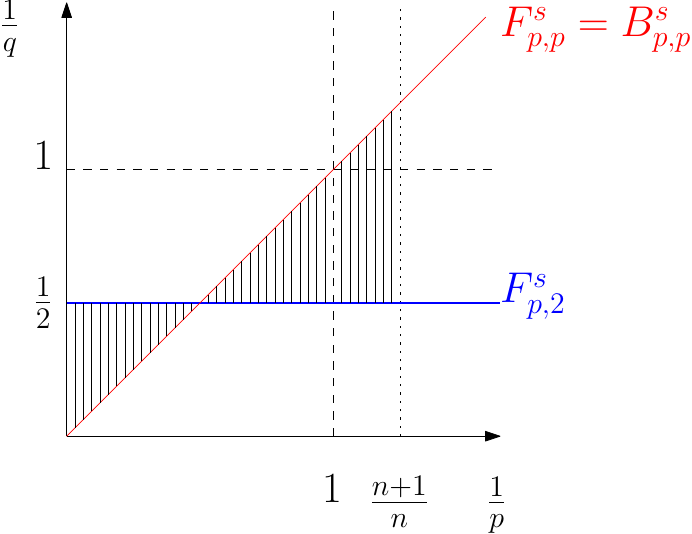}
\caption{Global boundedness in Triebel-Lizorkin scale.}
\end{figure}
\vspace*{.5cm}

This yields our main local and global boundedness results:

\begin{thm}\label{thm:local and global TL complement}
Let $0<p\leq \infty$, $\displaystyle m_c(p) := -(n-\nolinebreak 1)\abs {\frac 1p -\frac 12},$ $a (x,\xi)\in S^{m_c(p)}(\R^n)$, and $\varphi(x,\xi) \in\mathcal{C}^{\infty}(\R^n \times\R^n \setminus\{0\}),$ be positively homogeneous of degree one in $\xi$.  Then under these assumptions, the following results hold true$:$\\

\begin{enumerate}
\item[\emph{(i)}] If $a(x,\xi)$ has compact support in $x$ and $\varphi(x,\xi)$ is non-degenerate on the support of $a(x,\xi),$
then for any $s\in \R$,  $0<p< \infty$, $\min\, (2,p)\leq q\leq \max\,( 2,p)$, the operator $T_a^\varphi $
is bounded from $F_{p,q}^{s}(\R^n) $ to $ F_{p,q}^s(\R^n)$.\\

\item[\emph{(ii)}] If $\varphi (x,\xi)\in \Phi^2$ is $\mathrm{SND}$, then for any $s\in \R$,  $\displaystyle \frac{n}{n+1}<p< \infty$, \linebreak$\min\, (2,p)\leq q\leq \max\,( 2,p)$, the operator $T_a^\varphi$ is bounded from $F_{p,q}^{s}(\R^n) $ to $ F_{p,q}^s(\R^n)$.\\

\item[\emph{(iii)}] In both cases $\mathrm{(i)}$ and $\mathrm{(ii)}$ the corresponding operator is bounded from $F_{\infty,2}^{s}(\R^n) $ to $ F_{\infty,2}^{s}(\R^n),$ for $s\in \R$.\\
\item[\emph{(iv)}] If $\varphi(x,\xi)=|\xi|+x\cdot\xi$ for $s\in \mathbb{R}$ and $1\leq q\leq \infty$ one has that
for $a\in S^m(\R^n)$
\m{
\norm{ T_a^{|\cdot|}f}_{F^{s-m-\frac {n-1}2}_{1,q}(\R^n)}\lesssim \norm{ f}_{F^s_{1, q}(\R^n)}.
}
\end{enumerate}
\end{thm}

Statement $\mathrm{(iii)}$ is the consequence of the fact that for the aforementioned phases, the adjoint of the operator is bounded from $F_{1,2}^{-s}(\R^n) $ to $F_{1,2}^{-s}(\R^n).$\\

The last claim follows from the work of J. Peral \cite{Peral}, which implies that for \linebreak $\displaystyle a\in S^{-\frac{n-1}2}(\R^n)$ the operator $\displaystyle T_a^{|\cdot|}f$ has a factorisation $b(D)(f\ast d\sigma)$ where $d\sigma$ is the surface measure of the unit sphere and $b(\xi) \in S^0(\R^n)$. This and the Minkowski inequality in turn yield that

\m{
\norm{ T_a^{|\cdot|}f}_{F^{s-m-\frac {n-1}2}_{1,\infty}(\R^n)}\lesssim \Vert f\Vert_{F^s_{1, \infty}(\R^n)}.
}
and interpolation of this with $F^s_{1,2}(\R^n)$ yields the desired result.
\begin{rem}
The $F^s_{1, \infty}(\R^n)$ result above concerning the phase functions of the form  $x\cdot\xi+|\xi|$ could presumably be extended to a global result for phase functions of the form $x\cdot\xi+\phi(\xi)$ \emph{(}$\phi$ positively homogeneous of degree $1$\emph{)} or a local regularity for operators with phases of the form $\phi(x,\xi)$ \emph{(}positively homogeneous of degree $1$ in $\xi$ and non-degenerate\emph{)}. This is done by using a result of $\mathrm{T.\, Tao}$ \cite{tao} to decompose the corresponding \emph{FIOs} into a composition of a pseudodifferential operator and a generalised averaging operator $($which is bounded on $L^1(\R^n)$$)$. The details for this will appear elsewhere.
\end{rem}\hspace*{1cm}

\section{Boundedness of FIOs on Triebel-Lizorkin spaces in dimension one}\label{dimension_one}
In this section we separate the results in dimension one that were missing in the previous section for Triebel-Lizorkin spaces. We will also see that one has much more flexibility in dimension one in proving the optimal results for all scales of the Triebel-Lizorkin spaces. To this end we have
\begin{thm} Let $p\in (0,\infty]$, $a (x,\xi)\in S^{0}(\R)$ and $\varphi(x,\xi) \in\mathcal{C}^{\infty}(\R \times\R \setminus\{0\}),$ be positively homogeneous of degree one in $\xi$. 

If $\varphi\in \Phi^2$ and is $\mathrm{SND},$ then $T^{\varphi}_a$ is bounded from $F_{p,q}^s(\R)$ to itself, for $\displaystyle \frac{1}{2}<p\leq \infty$ and $0<q\leq \infty.$ Once again, the assumption of the compact support of the amplitude in $x$ and the non-degeneracy of the phase yields the result for the improved range $p\in (0,\infty]$.
\end{thm}

\begin{proof}
Let $a_+(\xi)\in C^{\infty}(\R)$ such that $a_+(\xi)=0$ when $\xi\leq 0$ and
 $a_+(\xi)=1$ when $\xi \geq 1$ and let $a_-(\xi):=a_+ (-\xi).$ Now write $1$ as $a_+(\xi)+a_-(\xi)+ r(\xi)$, where \linebreak$a_{\pm}\in S_{1,0}^{0}(\R)$
and $r(\xi)=1- a_+(\xi)-a_-(\xi) \in C^{\infty}_{c}(\R).$ Moreover using the (degree one) positive homogeneity of the phase function and the fact that we are in dimension one, we also have that $\varphi(x,\xi)=|\xi|\varphi(x,\sgn(\xi)).$ This yields that
\begin{equation*}
\begin{split}
T^{\phase}_{a}f(x)&= \int_{\R} a_{+}(x,\xi)\, e^{i\varphi(x,1)\xi}\, \widehat{f}(\xi) \ddd \xi + \int_{\R} a_{-}(x,\xi)\, e^{-i\varphi(x,-1)\xi} \,\widehat{f}(\xi) \ddd \xi \\ &+\int_{\R} r_a (x,\xi)\, e^{i\varphi(x,\xi)} \,\widehat{f}(\xi) \ddd \xi,
\end{split}
\end{equation*}
where $a_{\pm}(x,\xi)= a_{\pm}(\xi) \,a(x,\xi)$ and $r_a (x,\xi) = r(\xi) \,a(x,\xi).$
Therefore, using the invariance of $F^s_{p,q}(\R)$ (with $0<p<\infty$) under change of variables (observe that $\abs{\varphi'(x,1)}\lesssim 1$ by the $\Phi^2$ condition) and the boundedness of pseudodifferential operators on $F^s_{p,q}(\R)$ together with Proposition \ref{prop:low2 TL} above, we obtain the $F^s_{p,q}(\R)$ boundedness of the first two terms above. The boundedness of the third term is trivial as the amplitude of that operator belongs to $S^{-\infty}(\R)$.\\

For $F^s_{\infty,q}(\R)$, we use once again duality, which amounts to show that the adjoint operator
\m{
T^{\phase*}_{a}f(x) &= \iint_{\R\times\R} \overline{a}_{+}(y,\xi)\, e^{i(x-\varphi(y,1))\xi} f(y) \ddd\xi \dd y\\ &+ \iint_{\R\times\R} \overline{a}_{-}(y,\xi)\, e^{i(x+\varphi(y,-1))\xi} f(y) \ddd\xi \dd y\\ &+\iint_{\R\times\R} \overline{r}_a (y,\xi)\, e^{i(x\xi-\varphi(y,\xi))}  f(y) \ddd\xi \dd y.
}

is bounded from $F^{-s}_{1,q'}(\R)$ to itself where $\displaystyle \frac 1{q'} +\frac 1q =1$. Therefore, once again the invariance of $h^1(\R)$ under global diffeomorphisms with bounded Jacobians reduces the problem to show that a pseudodifferential operator of order zero the form
$$\displaystyle \iint_{\R\times\R} \overline{b}(y,\xi) \,e^{i(x-y)\xi} \, f(y) \ddd\xi \dd y $$
is bounded on $F^{-s}_{1,q'}(\R)$ which is well-known by e.g. \cite{Triebelpseudo}. The boundedness of the third term is trivial, once again due to the rapid decay of its amplitude. This concludes the proof of the theorem in the case of $0<p\leq \infty$ in dimension one.\\
\end{proof}
The following corollary yields the invariance of the Triebel-Lizorkin spaces $F^{s}_{\infty,q}(\R)$ under change of variables, which is missing in the literature, see e.g. Theorem \ref{invariance thm}.

\begin{cor}
If $\eta$ is a diffeomorphism from $\R$ to $\R$ such that $|\eta'(x)|\sim 1$ for all $x\in \mathbb{R}$ then for $0<q\leq \infty$ one has that
$$\Vert f\circ \eta\Vert_ {F^{s}_{\infty,q}(\R)}\lesssim \Vert f\Vert_{F^{s}_{\infty,q}(\R)}.$$
\end{cor}
\begin{proof}
The result follows by observing that $f\circ \eta(x)$ can be expressed as an FIO with amplitude  $1$ and the phase function $\eta(x)\xi$, which verifies both the SND and the $\Phi^2$ conditions and is therefore bounded on $F^{s}_{\infty,q}(\R)$.
\end{proof}\hspace*{1cm}
\section{Sharpness of the results}\label{section:sharpness}

\noindent In this section we explain why the restriction imposed on $p$ in Theorem \ref{thm:local and global BL} is necessary. To see this, if we let $\sigma\in S^{m_c(p)}(\R^n)$ be supported in a neighbourhood of the origin and take a function $f\in\S(\R^n)$ such that $\widehat{f}$ is equal to one on the support of $\sigma(\xi)$, and take $\psi_0\in C^{\infty}_{c}(\R^n)$ such that it is equal to one on the support of $\widehat{f}$.  Then we can take annuli-supported $\psi_j$'s such that
$\psi_j(D) T^{\phi}_{\sigma} f(x)=0$ for $j\geq 1$ and
$$\displaystyle \psi_0(D) T^{\phi}_{\sigma} f(x)= \int_{\R^n} \sigma (\xi)\, e^{ix\cdot\xi +i\phi(\xi)} \ddd \xi.$$
Now assume that $T^{\phi}_{\sigma}$ is bounded on $B^{s}_{p,q}(\R^n)$ for all $p\in (0,\infty]$ then
$$\norm{ T^{\phi}_{\sigma} f}_{B^{s}_{p,q}(\R^n)} \lesssim \Vert f\Vert_{B^{s}_{p,q}(\R^n)}.$$
Moreover using the boundedness assumption above, Definition \ref{def:Besov}, the fact that $\psi_j(D) T^{\phi}_{\sigma} f(x)=0$ for $j\geq 1$, and finally the frequency localisation of $f$ yield that for all $s, q$ and $p$ one has
\begin{equation*}
\begin{split}
\norm{\int_{\R^n} \sigma (\xi)\, e^{ix\cdot\xi +i\phi(\xi)} \ddd \xi}_{L^p(\R^n)}&=\Vert\psi_{0}(D)T^{\phi}_{\sigma} f\Vert_{L^p(\R^n)}=\Vert T^{\phi}_{\sigma} f\Vert_{B^{s}_{p,q}(\R^n)}\\ &\lesssim \Vert f \Vert_{B^{s}_{p,q}(\R^n)}= \Vert \psi_0 (D) f \Vert_{L^p(\R^n)}\\ & =\Vert f \Vert_{L^p(\R^n)}<\infty.
\end{split}
\end{equation*}
But since
$$\int_{\R^n} \sigma (\xi)\, e^{ix\cdot\xi +i\phi(\xi)} \ddd \xi$$
is equal to the convolution kernel $K(x)$ of the FIO $T^{\phi}_{\sigma}$, then the decay provided by Lemma \ref{lem:David-Wulf} which is actually sharp, won't yield $\displaystyle \Vert K \Vert_{L^p(\R^n)}\leq\Vert\langle\cdot\rangle^{-n-\varepsilon}\Vert_{L^p(\R^n)} <\infty,$ for $\displaystyle p\in \left (0, \frac{n}{n+1} \right ].$\\

\noindent In dimension $n=1$ we can explicitly see this by considering the FIO with amplitude identically equal to $1\in S^{0}(\R)$
\[
	Tf(x):=\int_{\R} \widehat{f}(\xi) e^{i \abs{\xi}+i x\xi}\ddd \xi=\frac{f(x+1)+f(x-1)}{2}+i\frac{Hf(x+1)-Hf(x-1)}{2},
\]
where the operator $H$ is the Hilbert transform. If we take $f$ to be the characteristic function of the interval $[-1,1]$, one can calculate that
\[
Hf(x) = \frac{1}{\pi}\log\left|\frac{x+1}{x-1}\right|.
\]
This implies that the imaginary part of $Tf$ is
\begin{align*}
\frac{Hf(x+1)-Hf(x-1)}{2} &= \frac{1}{2\pi}\left(\log\left|\frac{x+2}{x}\right| - \log\left|\frac{x}{x-2}\right|\right) = \frac{1}{2\pi}\log\left|1-\frac{4}{x^2}\right| \\
&= -\frac{2}{\pi x^2} + O\left (x^{-4} \right )
\end{align*}
as $|x|\to \infty$. Note that $\displaystyle \log\left|1-\frac{4}{x^2}\right|\in L^{p}_{\mathrm{loc}}(\R)$ for $0<p<\infty$, but since the real part of $Tf$ is compactly supported, the asymptotic expansion above yields that
\[
Tf(x) = \left(\frac{\widehat{f}(0)}{\pi i}\right) \frac{1}{x^2} + O\left (x^{-4} \right )
\]
as $|x|\to \infty$. From this, it follows that $Tf$ can not be in $B^{s}_{p,q}(\R)$ unless $\displaystyle p>\frac 12$.\\

\noindent The local result in Theorem \ref{thm:local and global BL} is sharp by the virtue of the sharpness of the classical Seeger-Sogge-Stein theorem \cite{SSS}.

\section{Applications to Hyperbolic PDEs}\label{section:applications in PDE}
\noindent In this section we outline some of the applications of the main results of this paper. This concerns local and global Besov-Lipschitz estimates for solutions to the Cauchy problems for strictly hyperbolic partial
differential equations.
\noindent First let us consider the basic example of the wave equation in $\R^{n+1}$
\begin{equation*}
     \left\{ \begin{array}{lll}
         \partial^2_t u (t,x)-\Delta_x u(t,x)=0, & t\not=0, \, x\in\R^n, \\
         u(0,x)=f_0 (x),\\
         \partial_t u(0,x)= f_1(x).
      \end{array} \right.
  \end{equation*}
It is well-known that the solution to this Cauchy problem is given by
$$u(t,x)= \int_{\R^n} e^{i(x\cdot\xi+ t|\xi|)}\left (\frac{\widehat{f}_{0}(\xi)}{2}+ \frac{\widehat{f}_{1}(\xi)}{2i|\xi|}\right ) \ddd \xi+ \int_{\R^n} e^{i(x\cdot\xi- t|\xi|)}\left(\frac{\widehat{f}_{0}(\xi)}{2}- \frac{\widehat{f}_{1}(\xi)}{2i|\xi|}\right ) \ddd \xi.$$
Now, using Theorem \ref{thm:local and global BL} it is not hard to verify that for some $\tau>0$ and each \linebreak$t\in[-\tau,\tau]$ and all $\displaystyle p\in \left (\frac{n}{n+1},\infty \right ]$, $0<q\leq \infty$, $m\in \R$, $s\in \R$ and \linebreak$\displaystyle m_c (p)$ as in \eqref{eq:criticaldecay} then
\m{
&\sup_{t\in [-\tau,\tau]}\left \Vert (1-\Delta)^{\frac m2} u \right \Vert_{B^{s}_{p,q}(\R^n)} \leq  C_\tau \left (  \Vert f_0\Vert_{B^{s+m-m_c (p)}_{p,q}(\R^n)}+ \Vert f_1\Vert_{B^{s+m-1-m_c (p)}_{p,q}(\R^n)}\right ),
}
from which it follows that the solution of the wave equation verifies the following global (spatial) Besov space estimate
\begin{equation}\label{main global besov estimate for the wave equation}
\sup_{t\in [-\tau,\tau]}\Vert u\Vert_{B^{s}_{p,q}(\R^n)}\leq C_\tau \left ( \Vert f_0\Vert_{B^{s+(n-1)\left |\frac{1}{p}-\frac{1}{2}\right |}_{p,q}(\R^n)}+ \Vert f_1\Vert_{B^{s+(n-1)\left |\frac{1}{p}-\frac{1}{2} \right |-1}_{p,q}(\R^n)}\right ).
\end{equation}
In particular, for $p=q$ and $s\in \R\setminus\mathbb{Z}$ (i.e. non-integer), \eqref{main global besov estimate for the wave equation} is the global extension of the Sobolev and Lipschitz space estimates in Theorem 4.1 of \cite{SSS}, for the case of wave equation. Moreover \eqref{main global besov estimate for the wave equation} goes beyond that result since it also provides estimates for the solution in quasi-Banach spaces.\\

Similarly, using Theorem \ref{thm:local and global TL complement}
 we have for any $s\in \R$,  $\displaystyle \frac{n}{n+1}<p<\infty$, \linebreak $\min\, (2,p)\leq q\leq \max\,( 2,p)$ that
 \begin{equation}\label{main global Triebel estimate for the wave equation}
\sup_{t\in [-\tau,\tau]}\Vert u\Vert_{F^{s}_{p,q}(\R^n)}\leq C_\tau \left ( \Vert f_0\Vert_{F^{s+(n-1)\left |\frac{1}{p}-\frac{1}{2}\right |}_{p,q}(\R^n)}+ \Vert f_1\Vert_{F^{s+(n-1)\left |\frac{1}{p}-\frac{1}{2} \right |-1}_{p,q}(\R^n)}\right ).
\end{equation}
Moreover if $p=1$ then the estimate above can actually be extended to the whole range $1\leq q\leq \infty$, and if $p=\infty$ and $q=2$ then the estimate still holds true, in particular one has
\begin{equation*}
\sup_{t\in [-\tau,\tau]}\Vert u\Vert_{\bmo(\R^n)}\leq C_\tau \left ( \Vert f_0\Vert_{F^{\frac{n-1}{2}}_{\infty,2}(\R^n)}+ \Vert f_1\Vert_{F^{\frac{n-3}{2}}_{\infty,2}(\R^n)}\right ),
\end{equation*}
 which yields that in 3 spatial dimensions,
 \begin{equation*}
\sup_{t\in [-\tau,\tau]}\Vert u\Vert_{\bmo(\R^3)}\leq C_\tau \left ( \Vert f_0\Vert_{F^{1}_{\infty,2}(\R^3)}+ \Vert f_1\Vert_{\bmo(\R^3)}\right ).
\end{equation*}\\

\noindent Concerning the local Besov space estimates, one can improve on the range of the estimates in $p$. In this connection let us consider the Cauchy problem for a strictly hyperbolic partial differential equation
\begin{equation}\label{hyp Cauchy prob}
  \begin{cases}
    \displaystyle Lu =0, \quad t\neq 0\\
    \left. \partial_{t}^{j} u \right|_{t=0}=f_ j (x),\quad 0\leq j\leq N-1,
  \end{cases}
\end{equation}
where $L:=\partial^{N}_{t} + \sum_{j=1}^{N} P_{j}(x,t,\nabla_x ) \partial_{t}^{N-j}$, $N\in \mathbb N$ and $P_j$ are variable-coefficient differential operators in such a way that $L$ becomes a strictly hyperbolic operator. This means that the principal symbol of $L$, denoted by $p(x, t, \xi, \tau)$ can be factored as 
\begin{equation}
p(x, t, \xi, \tau)=\prod_{j=1}^{m}\left(\tau-\lambda_{j}(x, t, \xi)\right),
\end{equation}
where all the $\lambda_{j}$'s are distinct, and are real {homogeneous symbols of degree one in $\xi$} that smoothly depend on $(x,t)$. \\

It is well-known (see e.g. \cite{Stein}) that this problem can be solved locally in time and modulo smoothing operators by
\begin{equation}\label{FIO representation of the solution}
  u(x,t)= \sum_{j=0}^{N-1}\sum_{k=1}^{N} \int_{\R^n} e^{i\varphi_{k}(x,\xi,t)} a_{jk}(x,\xi,t) \, \widehat{f_{j}}(\xi) \ddd \xi,
\end{equation}
where $a_{jk}(x,\xi,t)$ are suitably chosen amplitudes depending smoothly on $t$ and belonging to $S^{-j}_{1, 0}(\R^n)$, and the phases $\varphi_{k}(x,\xi,t)$ also depend smoothly on $t,$ are strongly non-degenerate and belong to the class $\Phi^2.$ This yields the following:
\begin{thm}\label{local besov estimate for hyperbolic pde}
\noindent Let $u(x,t)$ be the solution of the hyperbolic Cauchy problem \eqref{hyp Cauchy prob} with initial data $f_j$. Then for all $p, q\in (0, \infty]$ and $s\in \R$ and any $\chi\in \mathcal{C}^{\infty}_c (\R^n)$, the solution $u(\cdot, t)$ satisfies the local Besov-Lipschitz space estimate
  \begin{equation}\label{IRS local estimate}
    \sup_{t\in [-\tau,\tau]}\Vert \chi\, u\Vert_{B^{s}_{p,q}(\R^n)}\leq C_{\tau} \sum_{j=0}^{N-1}\Vert f_{j}\Vert_{B_{p,q}^{s+(n-1)\abs{\frac{1}{p}-\frac{1}{2}}-j}(\R^n)}.
  \end{equation}

  Similarly for any $s\in \R$,  $0<p< \infty$, $\min\, (2,p)\leq q\leq \max\,( 2,p)$, one has the local Triebel-Lizorkin estimate
  \begin{equation}
    \sup_{t\in [-\tau,\tau]}\Vert \chi\, u\Vert_{F^{s}_{p,q}(\R^n)}\leq C_{\tau} \sum_{j=0}^{N-1}\Vert f_{j}\Vert_{F_{p,q}^{s+(n-1)\abs{\frac{1}{p}-\frac{1}{2}}-j}(\R^n)},
  \end{equation}
  Which also holds when $p=\infty$ and $q=2$.
  Moreover if $\displaystyle m<-(n-1)\abs{\frac{1}{p}-\frac{1}{2}}$ then for all $s\in\mathbb{R}$ and $p, q\in (0, \infty]$ one has
  \begin{equation*}
    \sup_{t\in [-\tau,\tau]}\Vert \chi\, u\Vert_{F^{s}_{p,q}(\R^n)}\leq C_{\tau} \sum_{j=0}^{N-1}\Vert f_{j}\Vert_{F_{p,q}^{s-m-j}(\R^n)}.
  \end{equation*}
Furthermore, all the estimates above can be globalised  $($i.e. we can remove the cut-off function $\chi$ in all of them$)$ for $\displaystyle p\in \left ( \frac n{n+1} , \infty\right ]$, $q\in (0, \infty]$ and $s\in \R.$
\end{thm}

\begin{proof}
This follows at once from the Fourier integral operator representation \eqref{FIO representation of the solution} and theorems \ref{thm:local and global BL}, \ref{thm:local and global TL complement} and \ref{thm:local and global nonendpoint TL}.
\end{proof}
\noindent Estimate \eqref{IRS local estimate} is an extension of \eqref{SSS local sobolev estimate for the wave equation} which was proven in \cite{SSS}, to the case of $s\in\R$, $p\neq q$ and also the quasi-Banach setting.\\

\parindent 0pt

\bibsection

\end{document}